\documentclass[11pt]{amsart}

\usepackage[utf8]{inputenc}
\usepackage[english]{babel}
\usepackage{amsmath}
\usepackage{amssymb}
\usepackage[pdftex]{graphicx}
\usepackage{amsthm}
\usepackage{enumerate}
\usepackage[numbers,sort&compress]{natbib}
\usepackage{amsthm}
\usepackage[pdfborder={0 0 0}]{hyperref}
\usepackage{fge}
\usepackage{bm}
\usepackage{turnstile}
\usepackage{esvect}
\usepackage{color}
\usepackage{setspace} 
\usepackage{comment}

\newcommand{\comm}[1]{}

\DeclareRobustCommand{\SkipTocEntry}[5]{}

\newtheorem{satz}{Satz}[section]
\newtheorem{thm}[satz]{Theorem}

\newtheorem{lemma}[satz]{Lemma}
\newcounter{cl}[satz]
\newtheorem*{claim*}{Claim}
\newtheorem{claim}[cl]{Claim}
\newcounter{subcl}[cl]
\newtheorem{definition}[satz]{Definition}
\newtheorem{cor}[satz]{Corollary}

\newtheorem{question}{Question}

\newcommand{\thistheoremname}{}
\newtheorem*{genericthm}{\thistheoremname}

\newcounter{anhang}
\theoremstyle{definition}

\theoremstyle{remark}

\newtheorem*{remark}{Remark}

\numberwithin{subsection}{section}

\newcommand{\st}{\mid}

\newcommand{\ZFC}{\ensuremath{\operatorname{ZFC}} }
\newcommand{\ZF}{\ensuremath{\operatorname{ZF}} }
\newcommand{\GCH}{\ensuremath{\operatorname{GCH}} }
\newcommand{\CH}{\ensuremath{\operatorname{CH}} }

\newcommand{\Ord}{\ensuremath{\operatorname{Ord}} }
\newcommand{\HOD}{\ensuremath{\operatorname{HOD}} }
\newcommand{\HC}{\ensuremath{\operatorname{HC}} }
\newcommand{\OD}{\ensuremath{\operatorname{OD}} }

\newcommand{\TC}{\ensuremath{\operatorname{TC}} }

\newcommand{\AD}{\ensuremath{\operatorname{AD}} }
\newcommand{\crit}{\ensuremath{\operatorname{crit}} }

\newcommand{\Col}{\ensuremath{\operatorname{Col}} }
\newcommand{\lh}{\ensuremath{\operatorname{lh}} }

\newcommand{\Pot}{\mathcal{P}}

\newcommand{\forces}[2]{\Vdash^{#1}_{#2}}

\newcommand{\pwimg}{\, "}

\newcommand{\id}{\ensuremath{\operatorname{id}}}

\newcommand{\dom}{\ensuremath{\operatorname{dom}}}

\newcommand{\BS}{{}^\omega\omega}

\newcommand{\PI}{\boldsymbol\Pi}

\newcommand{\DELTA}{\boldsymbol\Delta}
\newcommand{\cQ}{\mathcal{Q}}
\newcommand{\cP}{\mathcal{P}}
\newcommand{\cM}{\mathcal{M}}
\newcommand{\cN}{\mathcal{N}}
\newcommand{\cT}{\mathcal{T}}
\newcommand{\cU}{\mathcal{U}}
\newcommand{\cF}{\mathcal{F}}
\newcommand{\cI}{\mathcal{I}}
\newcommand{\cR}{\mathcal{R}}
\newcommand{\cS}{\mathcal{S}}
\newcommand{\cG}{\mathcal{G}}

\newcommand{\bR}{\mathbb{R}}
\newcommand{\bP}{\mathbb{P}}
\newcommand{\bV}{\mathbb{V}}

\newcommand{\dlm}{\cM_{\infty}}

\newcommand{\tilextdlmp}{{\tilde\cM_{\infty}}^+}

\newcommand{\extdlm}{\tilde\cM_{\infty}}
\newcommand{\deltadlm}{\delta_{\infty}}
\newcommand{\kappadlm}{\kappa_{\infty}}
\newcommand{\dlmhat}{\hat{\cM}_{\infty}}

\newcommand{\HODMn}{\ensuremath{\operatorname{HOD}^{M_n(x)[g]}} }

\title{$\HOD$ in inner models with Woodin cardinals}



\author{Sandra M\"uller} 
\address{Sandra M\"uller, Institut f\"ur Mathematik, Universit\"at
  Wien. Kolingasse 14-16, 1090 Wien, Austria.}
\email{mueller.sandra@univie.ac.at} 
\thanks{Part of this
  work was done whilst the authors were visiting fellows at the Isaac
  Newton Institute for Mathematical Sciences in the programme
  \emph{Mathematical, Foundational and Computational Aspects of the
    Higher Infinite} (HIF) funded by EPSRC grant EP/K032208/1. The
  first author, formerly known as Sandra Uhlenbrock, was
  partially supported by FWF grant number P 28157 and in addition gratefully acknowledges funding from L'OR\'{E}AL Austria, in collaboration with the Austrian UNESCO Commission and in cooperation with the Austrian Academy of Sciences - Fellowship \emph{Determinacy and Large Cardinals}. The
  second author was partially supported by the NSF Career
  Award DMS-1352034.} 

\author{Grigor Sargsyan} 
\address{Grigor Sargsyan, Department of Mathematics,
  Rutgers University, Hill Center for the Mathematical Sciences, 110 Frelinghuysen Rd., Pisacataway, NJ 08854, USA} 
\email{grigor@math.rutgers.edu}

\date{\today}

\begin{document} 

\begin{abstract}
  We analyze the hereditarily ordinal definable sets $\HOD$ in
  $M_n(x)[g]$ for a Turing cone of reals $x$, where $M_n(x)$ is the
  canonical inner model with $n$ Woodin cardinals build over $x$ and
  $g$ is generic over $M_n(x)$ for the Lévy collapse up to its bottom
  inaccessible cardinal. We prove that assuming
  $\boldsymbol\Pi^1_{n+2}$-determinacy, for a Turing cone of reals
  $x$, $\HOD^{M_n(x)[g]} = M_n(\cM_{\infty} | \kappa_\infty, \Lambda),$
  where $\cM_\infty$ is a direct limit of iterates of $M_{n+1}$, $\delta_\infty$ is the least Woodin cardinal in $\cM_\infty$, $\kappa_\infty$ is the least inaccessible cardinal in $\cM_\infty$ above $\delta_\infty$, and $\Lambda$ is a partial iteration
  strategy for $\cM_{\infty}$. It will also be shown that under the
  same hypothesis $\HODMn$ satisfies $\GCH$.
\end{abstract}

\maketitle

\section{Introduction}
An essential question regarding the theory of inner models is the
analysis of the class of all hereditarily ordinal definable sets
$\HOD$ inside various inner models $M$ of the set theoretic universe
$V$ under appropriate determinacy hypotheses. Examples for such inner
models $M$ are $L(\mathbb{R})$, $L[x]$, and the canonical proper class
$x$-mouse with $n$ Woodin cardinals $M_n(x)$, but nowadays also larger
models of determinacy $M$ are considered.

One motivation for analyzing the internal structure of these models
$\HOD^M$ is given by Woodin's results in \cite{KW10} that under
determinacy hypotheses these models contain large cardinals. He showed
in \cite{KW10} for example that assuming $\Delta^1_2$ determinacy
there is a Turing cone of reals $x$ such that $\omega_2^{L[x]}$ is a
Woodin cardinal in the model $\HOD^{L[x]}$. This result generalizes to
higher levels in the projective hierarchy. That means for $n \geq 1$
assuming $\PI^1_{n+1}$ determinacy and $\Pi^1_{n+2}$ determinacy there
is a cone of reals $x$ such that $\omega_2^{M_{n}(x)}$ is a Woodin
cardinal in the model $\HOD^{M_{n}(x) | \delta_x}$, where $M_{n}(x)$
denotes the canonical proper class $x$-mouse with $n$ Woodin cardinals
and $\delta_x$ is the least Woodin cardinal in $M_{n}(x)$. Moreover,
Woodin showed a similar result for $\HOD^{L(\mathbb{R})}$. If we let
$\Theta$ denote the supremum of all ordinals $\alpha$ such that there
exists a surjection $\pi: \mathbb{R} \rightarrow \alpha$, then
assuming $\ZF + \AD$, he showed that $\Theta^{L(\mathbb{R})}$ is a
Woodin cardinal in $\HOD^{L(\mathbb{R})}$ (see \cite{KW10}). The fact
that these models of the form $\HOD^M$ can have large cardinals as for
example Woodin cardinals motivates the question if they are in some
sense fine structural as for example the models $L[x], M_n(x)$, and
$L(\mathbb{R})$ are. A good test question for this is whether these
models $\HOD^M$ satisfy the generalized continuum hypothesis
$\GCH$. If it turns out that $\HOD^M$ is in fact a fine structural
model, it would follow that it satisfies the $\GCH$ and even stronger
combinatorial principles as for example the $\Diamond$ principle.

The first model which was analyzed in this sense was
$\HOD^{L(\mathbb{R})}$ under the assumption that every set of reals in
$L(\mathbb{R})$ is determined (short: $\AD^{L(\mathbb{R})}$). Using
purely descriptive set theoretic methods Becker showed in \cite{Be80}
under this hypothesis that $\GCH_\alpha$, i.e. $2^\alpha = \alpha^+$,
holds in $\HOD^{L(\mathbb{R})}$ for all
$\alpha < \omega_1$. Later J. R. Steel and
W. H. Woodin were able to push the analysis of $\HOD^{L(\mathbb{R})}$
forward using more recent advances in inner model theory. In 1993 they
first showed independently that the reals in $\HOD^{L(\mathbb{R})}$
are the same as the reals in $M_\omega$, the least proper class
iterable premouse with $\omega$ Woodin cardinals. Then they showed in
§$4$ of \cite{St93} that $\HOD^{L(\mathbb{R})}$ in fact agrees with
the inner model $N$ up to $\Pot(\omega_1)$, where $N$
denotes the $\omega_1$-th linear iterate of $M_\omega$
by its least measure and its images. Building on this, John R. Steel
was able to show in \cite{St95} that $\HOD^{L(\mathbb{R})}$ agrees
with the inner model $\cM_\infty$ up to
$(\boldsymbol\delta_1^2)^{L(\mathbb{R})}$, where $\cM_\infty$ is a
direct limit of iterates of $M_\omega$ and
$(\boldsymbol\delta_1^2)^{L(\mathbb{R})}$ is the supremum of all
ordinals $\alpha$ such that there exists a surjection
$\pi: \mathbb{R} \rightarrow \alpha$ which is
$\DELTA_1^{L(\mathbb{R})}$ definable. Finally, in 1996 W. Hugh Woodin
extended this (see \cite{StW16}) and showed that in fact
$\HOD^{L(\mathbb{R})} = L[\cM_\infty, \Lambda]$, where $\Lambda$ is a
partial iteration strategy for $\cM_\infty$. For even larger models of
determinacy $M$ the corresponding model $\HOD^M$ was first analyzed in
\cite{Sa09}, where the second author showed that it is fine structural
using a layered hierarchy. Models of this form are nowadays called
\emph{hod mice}. A different approach for the fine structure of hod
mice called the least branch hierarchy is studied in \cite{St16}.

The question if $\HOD^{L[x]}$ is a model of $\GCH$ or even a fine structural model for a Turing cone of reals $x$ under a suitable determinacy hypothesis remains open until today. What has been done is the analysis of the model $\HOD^{L[x][G]}$, where $G$ is $\Col(\omega, {<}\kappa_x)$-generic over $\HOD^{L[x]}$ for the least inaccessible cardinal $\kappa_x$ in $L[x]$. Woodin showed in the $1990$'s (see \cite{StW16}) that assuming $\Delta^1_2$ determinacy there is a Turing cone of reals $x$ such that $\HOD^{L[x][G]} = L[\cM_\infty, \Lambda]$, where $\cM_\infty$ is a direct limit of mice (which are iterates of $M_1$) and $\Lambda$ is a partial iteration strategy for $\cM_\infty$.

In this article, we analyze $\HOD$ in the model $M_n(x)[g]$ for any real $x$ of sufficiently high Turing degree under the assumption that every $\PI^1_{n+2}$ set of reals is determined. Here $g$ is $\Col(\omega, {<}\kappa)$-generic over $M_n(x)$, where $\kappa$ denotes the least inaccessible cardinal in $M_n(x)$. We first show that the direct limit model $\dlm$, obtained from iterates of suitable premice, agrees up to its bottom Woodin cardinal $\deltadlm$ with $\HOD^{M_n(x)[g]}$. In a second step, we show that the full model $\HOD^{M_n(x)[g]}$ is in fact of the form $M_n(\dlmhat | \kappa_\infty, \Lambda)$, where $\dlmhat = M_n(\dlm | \deltadlm)$, $\kappa_\infty$ is the least inaccessible cardinal of $\dlmhat$ above $\deltadlm$, and $\Lambda$ is a partial iteration strategy for $\dlm$. Here and below $M_n(\dlmhat|\kappadlm, \Lambda)$ denotes the canonical fine structural model with $n$ Woodin cardinals build over the coarse objects $\dlmhat|\kappadlm$ and $\Lambda$. Our proof in fact shows that $\HOD^{M_n(x)[g]}$ is a model of $\GCH$, $\Diamond$, and other combinatorial principles which are consequences of fine structure.

In the statement of the following main theorem and in fact everywhere
in this article whenever we write $\HOD^M$ for some premouse $M$ we
mean $\HOD^{\lfloor M \rfloor}$, where $\lfloor M \rfloor$ denotes the
universe of the model $M$. In particular, we do not allow the extender
sequence of $M$ as a parameter in the definition of $\HOD$. It will be clear from the context if
we consider the model $M$ or the universe $\lfloor M \rfloor$ of $M$,
therefore we decided for the sake of readability to not distinguish
the notation for these two objects.

The main result of this paper is the following theorem.

\begin{thm}\label{thm:main}
  Let $n < \omega$ and assume
  $\boldsymbol\Pi^1_{n+2}$-determinacy. Then for a Turing cone of
  reals $x$,
  \[ \HOD^{M_n(x)[g]} = M_n(\dlmhat|\kappadlm, \Lambda), \] where $g$ is
  $\Col(\omega, {<}\kappa)$-generic over $M_n(x)$, $\kappa$ denotes
  the least inaccessible cardinal in $M_n(x)$,
  $\dlmhat$ is a direct limit of iterates of
  $M_{n+1}$, $\deltadlm$ is the least Woodin
  cardinal in $\dlmhat$, $\kappadlm$ is the least inaccessible cardinal of $\dlmhat$ above $\deltadlm$, and $\Lambda$ is a partial iteration strategy
  for $\dlm$.
\end{thm}

Our proof in fact shows the following corollary.

\begin{cor}\label{cor:main}
  Assume $\boldsymbol\Pi^1_{n+2}$-determinacy. Then for a Turing cone
  of reals $x$, \[\HODMn \vDash \GCH,\] where $g$ is
  $\Col(\omega, {<}\kappa)$-generic over $M_n(x)$ and $\kappa$ denotes
  the least inaccessible cardinal in $M_n(x)$.
\end{cor}

\begin{remark}
  In fact the full strength of $\PI^1_{n+2}$-determinacy is not needed
  for these results. It suffices to assume that $M_n^\#(x)$ exists and
  is $\omega_1$-iterable for all reals $x$ (or equivalently
  $\PI^1_{n+1}$-determinacy, see \cite{MSW20} and \cite{Ne02}) and that
  $M_{n+1}^\#$ exists and is $\omega_1$-iterable. This is all we will
  use in the proof.
\end{remark}

Finally, we summarize some open questions related to these results. The following question already appears in \cite{StW16}.

\begin{question}
  Assume $\DELTA^1_2$ determinacy. Is $\HOD^{L[x]}$ for a cone of reals
  $x$ a fine structural model?
\end{question}

\begin{question}
  Assume $\PI^1_{n+2}$ determinacy. Is $\HOD^{M_n(x)}$ for a cone of
  reals $x$ a fine structural model?
\end{question}

This article is structured as follows. In Section \ref{sec:prelim} we
recall some preliminaries and fix the basic notation. In Section
\ref{sec:dlm} we recall the relevant notions from \cite{Sa13} and
define the direct limit system converging to $\dlm$, before we compute
$\HODMn$ up to its Woodin cardinal in Section \ref{sec:dmr}. In
Section \ref{sec:fullHOD} we then show how this can be used to compute
the full model $\HODMn$, i.e., we finish the proof of Theorem
\ref{thm:main}. The authors thank Farmer Schlutzenberg for the helpful
discussions during the 4th Münster conference on inner model theory
in the summer of 2017. Finally, the authors thank the referee for
carefully reading the paper and making several helpful comments and
suggestions.

\section{Preliminaries and notation}\label{sec:prelim}

Whenever we say \emph{reals} we mean elements of the Baire space
$\BS$. We also write $\bR$ for $\BS$. $\HOD$ denotes the class of all
hereditarily ordinal definable sets. Moreover $\HOD_x$ for any
$x \in \BS$ denotes the class of all sets which are hereditarily
ordinal definable over $\{x\}$.\footnote{In the literature this is
  sometimes also called $\HOD_{\{x\}}$.} That means we let
$A \in \OD_x$ iff there is a formula $\varphi$ such that
$A = \{ v \st \varphi(v,\alpha_1, \dots, \alpha_n,x) \}$ for some
ordinals $\alpha_1, \dots, \alpha_n$. Then $A \in \HOD_x$ iff
$\TC(\{A\}) \subset \OD_x$, where $\TC(\{A\})$ denotes the transitive
closure of the set $\{A\}$.

We use the notions of premice and iterability from \cite[§$1-4$]{St10}
and assume that the reader is familiar with the basic concepts defined
there. In most cases we will demand $(\omega,\omega_1,\omega_1)$-iterability in the sense of Definition $4.4$ in \cite{St10} for our mice, but in other cases or if it is not clear from the context we will state the precise amount of iterability. We say a
\emph{cutpoint} of a premouse $\cM$ is an infinite ordinal $\gamma$
such that there is no extender $E$ on the $\cM$-sequence with
$\crit(E) \leq \gamma \leq \lh(E)$.\footnote{Such a cutpoint $\gamma$
  is often also called a strong cutpoint.} 

For some $\ZFC$ model $M$ and some real
$x \in M$ we write $L[E](x)^M$ for the result of a fully
backgrounded extender construction above $x$ inside $M$ in the sense
of \cite{MS94}, with the minimality condition relaxed to $\omega$-small premice. Moreover, we let for a premouse $\cM$ with $\cM \vDash \ZFC$, a cardinal cutpoint $\eta$ of $\cM$, and a premouse $\cN$ of height $\eta$ such that $\cN \in \Pot(\cM | \eta) \cap \cM | (\eta + \omega)$, $\cP^\cM(\cN)$ denote the
result of a $\cP$-construction over $\cN$ inside the model $\cM$ in
the sense of \cite{SchSt09} or \cite[Proposition 2.3 and Definition
2.4]{Sa13}.

For $x \in \BS$ and $n \leq \omega$ we let $M_n^\#(x)$, if it exists,
denote a countable, sound, $\omega_1$-iterable $x$-premouse which is
not $n$-small but all of whose proper initial segments are
$n$-small. In fact, $\omega_1$-iterability suffices to show that such
an $M_n^\#(x)$ is unique. If $M_n^\#(x)$ exists, we let $M_n(x)$ be
the proper class premouse obtained by iterating the top extender of
$M_n^\#(x)$ out of the universe.

\section{The direct limit system}\label{sec:dlm}

To show that $\HODMn$ is a fine structural inner model, we will use an
extension of the direct limit system introduced in \cite{Sa13}. For
the reader's convenience we will first recall the relevant definitions
and results from \cite{Sa13}, obtaining a direct limit system which is
definable in $M_n(x)$. We use the chance to correct some minor errors
in the presentation of that direct limit system in \cite{Sa13}. Then
we discuss the changes we need to make to obtain a direct limit system
definable in $M_n(x)[g]$. Another application of a similar but slightly
different direct limit system as in \cite{Sa13} can be found in
\cite{SaSch18}.

Fix an arbitrary natural number $n$. Throughout the rest of this article we will assume that $M_{n+1}^\#$ exists and is $(\omega,\omega_1,\omega_1)$-iterable and fix a real $x$ that codes $M_{n+1}^\#$. This implies $\PI^1_{n+1}$ determinacy or equivalently that
$M_{n}^\#(z)$ exists and is $(\omega, \omega_1, \omega_1)$-iterable for all reals $z$ (see
\cite{Ne95} and \cite{MSW20} for a proof of this equivalence due to Itay
Neeman and W. Hugh Woodin). Finally, we fix a $\Col(\omega, {<}\kappa)$-generic $g$ over $M_n(x)$, where $\kappa$ is the least inaccessible cardinal in $M_n(x)$.

\subsection*{The first direct limit system}

We first recall the definition of a lower part premouse.

\begin{definition}
  Let $a$ be a countable, transitive, self-wellordered\footnote{We say
  a transitive set $a$ is \emph{self-wellordered} iff $a$ is
  wellordered in $L_{\omega}[a]$.} set. Then we
define the lower part model $Lp^n(a)$ as the model theoretic union of all countable
$a$-premice $\cM$ with $\rho_\omega(\cM) = a$ which are $n$-small, sound,
and $(\omega, \omega_1, \omega_1)$-iterable. 
\end{definition}

If $\cN$ is a countable premouse, we also use $Lp^n(\cN)$ to denote
the premouse extending $\cN$ which is defined similarly as the model
theoretic union of premice $\cM \unrhd \cN$ with $\rho_\omega(\cM) \leq \cN \cap \Ord$ which have $\cN \cap \Ord$ as a cutpoint, are $n$-small above $\cN \cap \Ord$, sound above $\cN \cap \Ord$, and $(\omega, \omega_1, \omega_1)$-iterable above $\cN \cap \Ord$. In case $\cM \unrhd \cN$ has a partial measure $\mu$ with critical point $\cN \cap \Ord$, we replace $\cM$ by the corresponding translated premouse in which $\cN \cap \Ord$ is a cutpoint as in \cite[Remark 12.7]{St08} in order to include it in $Lp^n(\cN)$ as well.

\begin{definition} \label{def:nsuit} A countable premouse $\cN$ is
  \emph{$n$-suitable}\comm{Usually my notation for this is
    $(n+1)$-suitable, so be careful} iff there is an ordinal $\delta$
  such that
 \begin{enumerate}
 \item $\cN \vDash \text{``}\ZFC - \text{Replacement''}$ and
   $ \cN \cap \Ord = \sup_{i<\omega} (\delta^{+i})^\cN$,
 \item $\cN \vDash \text{``} \delta \text{ is a Woodin cardinal''}$,
 \item $\cN$ is $(n+1)$-small,
 \item for every cutpoint $\gamma < \delta$ of $\cN$, $\gamma$ is not
   Woodin in $Lp^n(\cN|\gamma)$,
 \item $\cN | (\delta^{+(i+1)})^\cN = Lp^n(\cN | (\delta^{+i})^\cN)$ for all $i<\omega$, and
 \item for all $\eta < \delta$, $\cN \vDash \text{``} \cN | \delta \text{ is } (\omega,\eta,\eta)\text{-iterable''}$. 
 \end{enumerate}
\end{definition}

If $\cN$ is an $n$-suitable premouse we denote the ordinal
$\delta$ from Definition \ref{def:nsuit} by
$\delta^\cN$. Moreover, we write $\hat\cN = M_n(\cN|\delta^{\cN})$ for any $n$-suitable premouse $\cN$. Then $\cN = \hat\cN | ((\delta^\cN)^{+\omega})^{\hat\cN}$ for every $n$-suitable premouse $\cN$ by well-known properties of the lower part model $Lp^n$. We now give some definitions indicating how $n$-suitable premice can be iterated. 

\begin{definition}
  Let $\cN$ be an arbitrary premouse and let $\cT$ be an iteration tree
  on $\cN$ of limit length.
  \begin{enumerate}[(1)]
  \item We say a premouse $\cQ = \cQ(\cT)$ is a \emph{$\cQ$-structure for
      $\cT$} iff $\cM(\cT) \unlhd \cQ$, $\cQ$ is sound above $\delta(\cT)$, $\delta(\cT)$ is a cutpoint of
    $\cQ$, $\cQ$ is $(\omega, \omega_1, \omega_1)$-iterable above $\delta(\cT)$, and if $\cQ \neq \cM(\cT)$ 
    \[ \cQ \vDash \text{``} \delta(\cT) \text{ is a Woodin
      cardinal''}, \] and 
    \begin{enumerate}[$(i)$]
    \item over $\cQ$ there exists an $r\Sigma_n$-definable set
      $A \subset \delta(\cT)$ such that there is no
      $\kappa < \delta(\cT)$ such that $\kappa$ is strong up to
      $\delta(\cT)$ with respect to $A$ as being witnessed by extenders
      on the sequence of $\cQ$ for some $n < \omega$, or
  \item  $\rho_n(\cQ) < \delta(\cT)$ for some $n<\omega$.
    \end{enumerate}
  \item Let $b$ be a cofinal well-founded branch through $\cT$. Then we
    say a premouse $\cQ = \cQ(b, \cT)$ is a \emph{$\cQ$-structure for $b$
      in $\cT$} iff 
    $\cQ = \cM_b^\cT | \gamma$, where $\gamma \leq \cM_b^\cT \cap \Ord$ is
    the least ordinal such that either
    \[ \gamma < \cM_b^\cT \cap \Ord \text{ and } \cM_b^\cT | (\gamma+1)
    \vDash \text{``}\delta(\cT) \text{ is not Woodin'',} \] or
    \[ \gamma = \cM_b^\cT \cap \Ord \text{ and }
    \rho_n(\cM_b^\cT) < \delta(\cT) \]
    for some $n<\omega$ or over $\cM_b^\cT$ there exists an
    $r\Sigma_n$-definable set $A \subset \delta(\cT)$ such that there
    is no $\kappa < \delta(\cT)$ such that $\kappa$ is strong up to
    $\delta(\cT)$ with respect to $A$ as being witnessed by extenders
    on the sequence of $\cM_b^\cT$ for some $n < \omega$.\\
    If no such ordinal $\gamma \leq \cM_b^\cT \cap \Ord$ exists, we let
    $\cQ(b,\cT)$ be undefined.
    \end{enumerate}
  \end{definition}

  \begin{remark}
  The premouse $M_{n+1}|(\delta_0^{+\omega})^{M_{n+1}}$ is
  $n$-suitable, where $\delta_0$ is the least Woodin cardinal in
  $M_{n+1}$. We denote this premouse by $M_{n+1}^-$ and write
  $\Sigma_{M_{n+1}^-}$ for its canonical iteration strategy induced by
  the usual iteration strategy $\Sigma_{M_{n+1}}$ for $M_{n+1}$ for countable stacks of normal trees without drops on the main branches.
\end{remark}

Our goal is to approximate the iteration strategy
$\Sigma_{M_{n+1}^-}$ inside $\HODMn$. Analogous to \cite[Definition 5.32]{SchlTr} we define the following requirement, which will be used in Definition \ref{def:sti} to make
the proof of Lemmas \ref{lem:pseudocomp} and \ref{lem:pseudogenit} work.

\begin{definition}
  Let $\cN$ be an $n$-suitable premouse and let $\cT$ be a normal
  iteration tree on $\cN$ of length $< \omega_1^V$. Then we say that
  $\cT$ is \emph{suitability strict} iff for all $\alpha < \lh(\cT)$,
  \begin{enumerate}[$(i)$]
  \item if $[0, \alpha]_T$ does not drop then $\cM_\alpha^\cT$ is
    $n$-suitable, and 
  \item if $[0, \alpha]_T$ drops then no $\cR \unlhd \cM_\alpha^\cT$
    is $n$-suitable. 
  \end{enumerate}
\end{definition}

\begin{definition}
  \label{def:short} Let $\cN$ be an $n$-suitable premouse and let
  $\cT$ be a normal iteration tree on $\cN$ of length $< \omega_1^V$.
  \begin{enumerate}
  \item $\cT$ is \emph{correctly guided} iff for every limit ordinal
    $\lambda < \lh(\cT)$, if $b$ is the branch choosen for
    $\cT \upharpoonright \lambda$ in $\cT$, then
    $\cQ(b, \cT \upharpoonright \lambda)$ exists and
    $\cQ(b, \cT \upharpoonright \lambda) \unlhd M_n(\cM(\cT
    \upharpoonright \lambda))$.
  \item $\cT$ is \emph{short} iff $\cT$ is correctly guided and in
    case $\cT$ has limit length $\cQ(\cT)$ exists and $\cQ(\cT) \unlhd
    M_n(\cM(\cT))$.
  \item $\cT$ is \emph{maximal} iff $\cT$ is correctly guided and not
    short.
  \end{enumerate}
\end{definition}

\begin{definition}\label{def:sti}
  Let $\cN$ be an $n$-suitable premouse. We say $\cN$ is \emph{short
    tree iterable} \index{short tree iterable} iff whenever $\cT$ is a
  short tree on $\cN$,
  \begin{enumerate}[$(i)$]
  \item $\cT$ is suitability strict,
  \item if $\cT$ has a last model, then every putative\footnote{An
      iteration tree $\cU$ is a \emph{putative iteration tree} if
      $\cU$ satisfies all properties of an iteration tree, but in case
      $\cU$ has a last model we allow this last model to be
      ill-founded.} iteration tree $\cU$ extending $\cT$ such that
    $\lh(\cU) = \lh(\cT) +1$ has a well-founded last model, and
  \item if $\cT$ has limit length, then there exists a cofinal
    well-founded branch $b$ through $\cT$ such that
    $\cQ(b,\cT) = \cQ(\cT)$.
  \end{enumerate}
\end{definition}

This can be generalized to stacks of correctly guided normal trees.

\begin{definition}\label{def:fincorrectlyguidedstack}
  Let $\cN$ be an $n$-suitable premouse and $m < \omega$. Then we say
  $(\cT_i, \cN_i \st i \leq m)$ is a \emph{correctly guided finite
    stack} on $\cN$ iff
  \begin{enumerate}[$(i)$]
  \item $\cN_0 = \cN$,
  \item $\cN_i$ is $n$-suitable and $\cT_i$ is a correctly guided
    normal iteration tree on $\cN_i$ which acts below $\delta^{\cN_i}$
    for all $i \leq m$,
  \item for every $i<m$ either $\cT_i$ has a last model which is equal
    to $\cN_{i+1}$ and the iteration embedding
    $i^{\cT_i}: \cN_i \rightarrow \cN_{i+1}$ exists or $\cT_i$ is
    maximal and
    $\cN_{i+1} = M_n(\cM(\cT_i)) |
    (\delta(\cT_i)^{+\omega})^{M_n(\cM(\cT_i))}$.
  \end{enumerate}
  Moreover, we say that $\cM$ is the last model of
  $(\cT_i, \cN_i \st i \leq m)$ iff either
  \begin{enumerate}[$(i)$]
  \item $\cT_m$ has a last model which is equal to $\cM$ and the
    iteration embedding $i^{\cT_m}: \cN_m \rightarrow \cM$ exists,
  \item $\cT_m$ is of limit length and short and there is a
    non-dropping cofinal well-founded branch $b$ through $\cT_m$ such
    that $\cQ(b,\cT)$ exists, $\cT_m{}^\smallfrown b$ is correctly
    guided, and $\cM = \cM_b^{\cT}$, or
  \item $\cT_m$ is maximal and
    $\cM = M_n(\cM(\cT_m)) |
    (\delta(\cT_m)^{+\omega})^{M_n(\cM(\cT_m))}$.
  \end{enumerate}
  Finally, we say that $\cM$ is a \emph{correct iterate} of $\cN$ iff
  there is a correctly guided finite stack on $\cN$ with last model
  $\cM$. In case there is a correctly guided finite stack on $\cN$
  with last model $\cM$ of length $1$, i.e., such that $m=0$, we say
  that $\cM$ is a \emph{pseudo-normal iterate} (or just
  \emph{pseudo-iterate}) of $\cN$.
\end{definition}

Analogous to Theorem $3.14$ in \cite{StW16} we also have a version of
the comparison lemma for short tree iterable premice and pseudo-normal
iterates.

\begin{lemma}[Pseudo-comparison lemma]\label{lem:pseudocomp}
  Let $\cN$ and $\cM$ be $n$-suitable premice which are short tree
  iterable. Then there is a common pseudo-normal iterate
  $\cR \in M_n(y)$ such that
  $\delta^\cR \leq \omega_1^{M_n(y)}$, where
  $y$ is a real coding $\cN$ and $\cM$.
\end{lemma}

The proof of Lemma \ref{lem:pseudocomp} is similar to the proof of
Theorem $3.14$ in \cite{StW16}, so we omit it. Similarly, we have an
analogue to the pseudo-genericity iteration (see Theorem $3.16$ in
\cite{StW16}).

\begin{lemma}[Pseudo-genericity iterations]\label{lem:pseudogenit}
  Let $\cN$ be an $n$-suitable premouse which is short tree iterable
  and let $z$ be a real. Then there is a pseudo-normal iterate $\cR$
  of $\cN$ in $M_n(y,z)$ such that $z$ is $\mathbb{B}^\cR$-generic
  over $\cR$ and $\delta^\cR \leq \omega_1^{M_n(y,z)}$, where
  $y$ is a real coding $\cN$ and $\mathbb{B}^\cR$ denotes Woodin's
  extender algebra inside $\cR$.
\end{lemma}

For the definition of the direct limit system converging to $\HOD$ we
need the notion of $s$-iterability. To define this, we first introduce
some notation.
For an $n$-suitable premouse $\cN$, a finite sequence of ordinals $s$,
and some $k < \omega$ let
\begin{align*}
  T_{s,k}^\cN = \{ (t, \ulcorner \phi \urcorner) \in 
  [((\delta^\cN)^{+k})^\cN]^{<\omega} & \times \omega \st \phi \text{ is
                                        a } \Sigma_1\text{-formula and
                                        }\\ & M_n(\cN | \delta^\cN)
                                              \vDash \phi[t,s] \},  
\end{align*}
where $\ulcorner \phi \urcorner$ denotes the G\"odel number of
$\phi$. Let $Hull_1^\cN$ denote an uncollapsed $\Sigma_1$ hull in
$\cN$. Then we let
\[ \gamma_s^\cN = \sup(Hull_1^\cN(\{ T_{s,k}^\cN \st k < \omega \}) \cap
  \delta^\cN)\] and
\[ H_s^\cN = Hull_1^\cN( \gamma_s^\cN \cup \{ T_{s,k}^\cN \st k <
  \omega \}). \] Then $\gamma_s^\cN = H_s^\cN \cap \delta^\cN$. For
$s_m = (u_1, \dots, u_{m})$ the sequence of the first $m$ uniform
indiscernibles, we write $\gamma_m^\cN = \gamma_{s_m}^\cN$ and
$H_m^\cN = H_{s_m}^\cN$. Then we have that
$\sup_{m \in \omega} \gamma_m^\cN = \delta^\cN$ (see Lemma 5.3 in
\cite{Sa13}).

\begin{definition}\label{def:siterable}
  Let $\cN$ be an $n$-suitable premouse and $s$ a finite sequence of
  ordinals. Then $\cN$ is \emph{$s$-iterable} iff every correct iterate of $\cN$ is short tree
  iterable and for every correctly guided finite stack
  $(\cT_i, \cN_i \st i \leq m)$ on $\cN$ with last model $\cM$ there
  is a sequence of non-dropping branches $(b_i \st i \leq m)$ and a sequence of
  embeddings $(\pi_i \st i \leq m)$ such that
  \begin{enumerate}[$(i)$]
  \item if $\cT_i$ has successor length $\alpha +1$, then
    $b_i = [0,\alpha]_{T_i}$ and $\pi_i = i_{0,\alpha}^{\cT_i}$ is the corresponding
    iteration embedding for $i \leq m$,
  \item if $\cT_m$ is short, then $b_m$ is the unique cofinal
    well-founded branch through $\cT_m$ such that $\cQ(b_m,\cT_m)$
    exists and $\cT_m{}^\smallfrown b_m$ is correctly guided and
    $\pi_m = i^{\cT_m}_{b_m}$ is the corresponding iteration
    embedding,
  \item if $\cT_i$ is maximal, then $b_i$ is a cofinal well-founded
    branch through $\cT_i$ such that $\cM_{b_i}^{\cT_i} = \cN_{i+1}$
    if $i<m$ or $\cM_{b_i}^{\cT_i} = \cM$ if $i =m$, and
    $\pi_i = i^{\cT_i}_{b_i}$ is the corresponding iteration embedding
    for $i \leq m$, and
  \item if we let
    $\pi = \pi_m \circ \pi_{m-1} \circ \dots \circ \pi_0$ then for
    every $k< \omega$, \[\pi(T_{s,k}^\cN) = T_{s,k}^\cM.\]
  \end{enumerate}
  In this case we say that the sequence $\vec{b} = (b_i \st i \leq m)$
  \emph{witnesses $s$-iterability} for
  $\vec{\cT} = (\cT_i, \cN_i \st i \leq m)$ or that $\vec{b}$ is an
  \emph{$s$-iterability branch} for $\vec{\cT}$ and we write
  $\pi_{\vec{\cT}, \vec{b}} = \pi$.
\end{definition}

Now for every two $s$-iterability branches for $\vec{\cT}$ on $\cN$
their corresponding iteration embeddings agree on $H_s^\cN$.

\begin{lemma}[Uniqueness of $s$-iterability embeddings, Lemma $5.5$ in
  \cite{Sa13}]\label{lem:uniqsitemb}
  Let $\cN$ be an $n$-suitable premouse, $s$ a finite sequence of
  ordinals, and $\vec{\cT}$ a correctly guided finite stack on
  $\cN$. Moreover let $\vec{b}$ and $\vec{c}$ be $s$-iterability
  branches for $\vec{\cT}$. Then
  \[ \pi_{\vec{\cT}, \vec{b}} \upharpoonright H_s^\cN =
    \pi_{\vec{\cT}, \vec{c}} \upharpoonright H_s^\cN. \]
\end{lemma}

The uniqueness of $s$-iterability embeddings yields that for every $n$-suitable, $s$-iterable
$\cN$, every correctly guided finite stack $\vec{\cT}$ on $\cN$ and
every $s$-iterability branch $\vec{b}$ for $\vec{\cT}$, the embedding
$\pi_{\vec{\cT}, \vec{b}} \upharpoonright H_s^\cN$ is independent of
the choice of $\vec{b}$, but it might still depend on
$\vec{\cT}$. This motivates the following definition.

\begin{definition}\label{def:stronglysiterable}
  Let $\cN$ be an $n$-suitable premouse and $s$ a finite sequence of
  ordinals. Then $\cN$ is \emph{strongly $s$-iterable} iff for every correct iterate $\cR$ of $\cN$, $\cR$ is $s$-iterable and for every two correctly guided finite stacks
  $\vec{\cT}$ and $\vec{\cU}$ on $\cR$ with common last model $\cM$
  and $s$-iterability witnesses $\vec{b}$ and $\vec{c}$ for
  $\vec{\cT}$ and $\vec{\cU}$ respectively, we have that
  \[ \pi_{\vec{\cT}, \vec{b}} \upharpoonright H_s^\cR =
    \pi_{\vec{\cU}, \vec{c}} \upharpoonright H_s^\cR. \]
\end{definition}

A so-called \emph{bad sequence argument} shows the following lemma,
which yields the existence of strongly $s$-iterable premice.

\begin{lemma}[Lemma $5.9$ in \cite{Sa13}]\label{lem:stronglysiterable}
  For every finite sequence of ordinals $s$ and any short tree
  iterable $n$-suitable premouse $\cN$ there is a pseudo-normal
  iterate $\cM$ of $\cN$ such that $\cM$ is strongly $s$-iterable.
\end{lemma}

If $\cN$ is strongly $s$-iterable and $\vec{\cT}$ is a correctly
guided finite stack on $\cN$ with last model $\cM$, let
$\pi_{\cN,\cM,s} : H_s^\cN \rightarrow H_s^\cM$ denote the embedding
given by any $s$-iterability branch $\vec{b}$ for $\vec{\cT}$. As
$\cN$ is strongly $s$-iterable, the embedding $\pi_{\cN,\cM,s}$ does
not depend on the choice of $\vec{\cT}$ and $\vec{b}$.

Recall that we write
$M_{n+1}^- = M_{n+1} | (\delta_0^{+\omega})^{M_{n+1}}$, where
$\delta_0$ is the least Woodin cardinal in $M_{n+1}$, and
$\Sigma_{M_{n+1}^-}$ for the canonical iteration strategy for
$M_{n+1}^-$ induced by $\Sigma_{M_{n+1}}$. Moreover, recall that for $m<\omega$, we write $s_m$ for the set of the first $m$ uniform indiscernibles. Then $M_{n+1}^-$ is
$n$-suitable and strongly $s_m$-iterable for every $m$. Moreover, if
$\vec{\cT}$ is a correctly guided finite stack on $M_{n+1}^-$ with
last model $\cM$, then $\pi_{M_{n+1}^-, \cM,s_m}$ agrees with the
iteration embedding according to $\Sigma_{M_{n+1}^-}$ on $H_{s_m}^{M_{n+1}^-}$. The first direct limit
system we define will consist of iterates of $M_{n+1}^-$.

\begin{definition}
  Let
  \begin{align*}
     {\tilde\cF}^+ = \{ \cN \st \cN \text{ is} & \text{ a non-dropping iterate of } M_{n+1}^- \text{ via } \Sigma_{M_{n+1}^-} \\ & \text{ by a finite stack of countable trees}\}
  \end{align*}
 and for $\cN, \cM \in {\tilde\cF}^+$ let
  $\cN \leq^+ \cM$ iff $\cM$ is an iterate of $\cN$ via the tail
  strategy $\Sigma_\cN$ as witnessed by some finite stack of countable iteration
  trees. Then we let $\tilextdlmp$ be the direct limit of
  $({\tilde\cF}^+, \leq^+)$ under the iteration maps. \comm{Moreover for
  $\cN \in {\tilde}\cF^+$ let $i_{\cN,\infty} : \cN \rightarrow \tilextdlmp$
  denote the corresponding direct limit embedding.}
\end{definition}

\begin{remark}
  The prewellordering $\leq^+$ on ${\tilde\cF}^+$ is directed and the direct
  limit $\tilextdlmp$ is well-founded as the limit system
  $({\tilde\cF}^+, \leq^+)$ only consists of iterates of $M_{n+1}^-$ via the
  canonical iteration strategy $\Sigma_{M_{n+1}^-}$.
\end{remark}

Since ${\tilde\cF}^+$ is not definable enough for our purposes, we now
introduce another direct limit system which has the same direct limit
$\tilextdlmp$.

\begin{definition}
  Let
  \[ \tilde\cI = \{ (\cN, s) \st \cN \text{ is } n\text{-suitable}, s \in
    [\Ord]^{<\omega}, \text{ and } \cN \text{ is strongly }
    s\text{-iterable} \} \] and
  \[ \tilde\cF = \{ H_s^\cN \st (\cN,s) \in \tilde\cI \}. \] For
  $(\cN,s), (\cM,t) \in \tilde\cI$ we let $(\cN,s) \leq_{\tilde\cI} (\cM,t)$ iff
  there is a correctly guided finite stack on $\cN$ with last model
  $\cM$ and $s \subseteq t$. In this case we let
  $\pi_{(\cN,s), (\cM,t)}:H_s^{\cN} \rightarrow H_t^{\cM}$ denote
  the canonical corresponding embedding.
\end{definition}

\begin{remark}
  The prewellordering $\leq_{\tilde\cI}$ on $\tilde\cI$ is directed: Let
  $(\cN,s), (\cM,t) \in \tilde\cI$. By Lemma \ref{lem:stronglysiterable}
  there exists an $n$-suitable premouse $\cR$ which is strongly
  $(s \cup t)$-iterable. Let $\cS$ be the result of simultaneously
  comparing $\cN$, $\cM$ and $\cR$ in the sense of Lemma
  \ref{lem:pseudocomp}. Then $(\cS, s \cup t) \in \tilde\cI$,
  $(\cN, s) \leq_{\tilde\cI} (\cS, s \cup t)$, and
  $(\cM, t) \leq_{\tilde\cI} (\cS, s \cup t)$, as desired.
\end{remark}

\begin{definition}
  Let $\extdlm$ be the direct limit of $(\tilde\cF, \leq_{\tilde\cI})$ under the
  embeddings $\pi_{(\cN,s), (\cM,t)}$. For $(\cN, s) \in \tilde\cI$ let
  $\pi_{(\cN,s),\infty} : H_s^\cN \rightarrow \extdlm$ denote the
  corresponding direct limit embedding.
\end{definition}

The fact that $\extdlm$ is well-founded follows from the next lemma.

\begin{lemma}[Lemma $5.10$ in \cite{Sa13}]\label{lem:dlmequal}
  $\extdlm = \tilextdlmp$.
\end{lemma}

\subsection*{The second direct limit system}

To obtain $\HOD$ of some inner model from the direct limit, we in
particular need to show that the direct limit is in fact contained
in $\HOD$ of that inner model. In our setting we therefore need to
internalize the direct limit system into the inner model $M_n(x)[g]$ fixed above. We first aim to define a direct limit system similar to
$(\tilde\cF, \leq_{\tilde\cI})$ in $M_n(x)$ analogous to \cite{Sa13}. In
a second step, we then modify the system to obtain direct limit
systems with the same direct limit which are definable in
$M_n(x)[g]$.

The notion of $n$-suitability from Definition \ref{def:nsuit} is
already internal to $M_n(x)$ and $M_n(x)[g]$, i.e., if $\cN \in M_n(x)|\kappa$
then $\cN$ is $n$-suitable in $V$ iff $\cN$ is $n$-suitable in
$M_n(x)$ by the following lemma.

\begin{lemma}\label{lem:closureSn} Let $\delta_0$ denote the least Woodin cardinal in $M_n(x)$.
  \begin{enumerate}
  \item For all $y \in V^{M_n(x)[g]}_{\delta_0}$, $Lp^n(y) \in \HOD_y^{M_n(x)[g]}$.
  \item $V^{M_n(x)}_{\delta_0}$ and $V^{M_n(x)[g]}_{\delta_0}$ are closed under the operation $y \mapsto Lp^n(y)$. 
  \end{enumerate}
\end{lemma}

\begin{proof}
  Let $y \in V^{M_n(x)[g]}_{\delta_0}$ be arbitrary. The model $M_n(x)[g]$ can be organized as a $V_\kappa^{M_n(x)[g]}$-premouse and as such it inherits the iterability from $M_n(x)$ and is in fact equal to $M_n(V_\kappa^{M_n(x)[g]})$. Consider $L[E](y)^{M_n(x)[g]}$, the result of a fully backgrounded extender construction above $y$ using extenders from the sequence of $M_n(x)[g]$ organized as a $V_\kappa^{M_n(x)[g]}$-premouse, and compare it with 
 $Lp^n(y)$. First, we argue that $Lp^n(y)$ does not move. If it would move, the $Lp^n(y)$-side of the coiteration would have to drop because $Lp^n(y)$ does not have any total extenders. Moreover, it would have to iterate to a proper class model which is equal to an iterate of $L[E](y)^{M_n(x)[g]}$. As $L[E](y)^{M_n(x)[g]}$ has $n$ Woodin cardinals, this would imply that $Lp^n(y)$ has a level which is not $n$-small, contradicting the definition of $Lp^n(y)$.

 Therefore, $Lp^n(y) \lhd \cR$ for some iterate $\cR$ of $L[E](y)^{M_n(x)[g]}$.\footnote{In fact, it is not hard to see that $\cR = L[E](y)^{M_n(x)[g]}$ and hence $Lp^n(y) \lhd L[E](y)^{M_n(x)[g]}$ but we will not need this observation.} The iteration from $L[E](y)^{M_n(x)[g]}$ to $\cR$ resulting from the comparison process can be defined over $L[E](y)^{M_n(x)[g]}$ from the extender sequence of $L[E](y)^{M_n(x)[g]}$ and a finite sequence of ordinals as it cannot leave any total measures behind and thus can only use measures of order $0$. Let $K(V_\kappa^{M_n(x)[g]})^{M_n(V_\kappa^{M_n(x)[g]})}$ denote the core model constructed above $V_\kappa^{M_n(x)[g]}$ inside $M_n(V_\kappa^{M_n(x)[g]})$, an $n$-small premouse with $n$ Woodin cardinals, in the sense of \cite{Sch06}. Then $K(V_\kappa^{M_n(x)[g]})^{M_n(V_\kappa^{M_n(x)[g]})} = M_n(V_\kappa^{M_n(x)[g]})$ by \cite[Lemma 1.1 (due to J. Steel)]{Sch06}. Recall that the reorganizsation of $M_n(x)[g]$ as a $V_\kappa^{M_n(x)[g]}$-premouse is equal to $M_n(V_\kappa^{M_n(x)[g]})$. Hence, the extender sequence of the $V_\kappa^{M_n(x)[g]}$-premouse $M_n(x)[g]$ is in $\OD^{M_n(x)[g]}_{V_\kappa^{M_n(x)[g]}} = \OD^{M_n(x)[g]}$. Therefore $Lp^n(y) \in \HOD^{M_n(x)[g]}$ by the definability of the $L[E]$-construction.

  For $(2)$, the closure of $V^{M_n(x)[g]}_{\delta_0}$ follows immediately from $(1)$. For $V^{M_n(x)}_{\delta_0}$ notice that for $y \in V^{M_n(x)}_{\delta_0}$, $Lp^n(y) \in \HOD_y^{M_n(x)[g]} \subseteq \HOD_y^{M_n(x)}$ by homogeneity of the forcing.
\end{proof}

By \emph{stacking} the $Lp^n$-operation, the uniform proof of Lemma \ref{lem:closureSn} in fact shows that for all $y \in V_{\delta_0}^{M_n(x)}$, $M_n(y)|\kappa_0 \in \HODMn_y$, where $\kappa_0$ denotes the least measurable cardinal in $M_n(y)$.  

The definitions of short tree, maximal tree, and correctly guided finite stack we gave above are internal to $M_n(x)$ and $M_n(x)[g]$ as well, as they can be defined only using the $Lp^n$-operation. The only notion we have to take care of is $s$-iterability since it is not even clear how the sets $T_{s,k}^\cN$ can be identified inside $M_n(x)$. This obstacle is solved by shrinking the direct limit system $(\tilde\cF, \leq_{\tilde\cI})$ to a dense subset as follows. 

\begin{definition}
  Let
  \begin{align*}
    \cG & = \;   \{ \cN \in M_n(x) | \kappa \st \cN \text{ is }
                   n\text{-suitable and } M_n(x) \vDash \text{``for
                      some cardinal}\\
                 & \; \text{cutpoint } \eta,  
                   \delta^\cN = \eta^+, \cN | \delta^\cN \in \cP(M_n(x) | \eta^+) \cap M_n(x) | (\eta^+ + \omega), \\
                &  \;  \text{and } M_n(x) | \eta \text{ is generic
                   over } \cN \text{ for the } \delta^\cN\text{-generator version of} \\ & \; \; \; \; \; \; \text{the extender algebra at } \delta^\cN \text{''} \}.
  \end{align*}
\end{definition}

See for example Section $4.1$ in \cite{Fa} for an introduction to the
$\delta$-generator version of the extender algebra at some Woodin
cardinal $\delta$. 
The following lemma shows how we can use the fact that $\cN \in \cG$ to
detect $M_n(\cN | \delta^{\cN})$ inside $M_n(x)$. For some premouse
$\cR \in \cG$ we denote the last model of a $\cP$-construction
above $\cR|\delta^\cR$ performed inside $M_n(x)$ as introduced in \cite{SchSt09}
(see also Proposition $2.3$ and Definition $2.4$ in \cite{Sa13}) by
$\cP^{M_n(x)}(\cR|\delta^\cR)$.

\begin{lemma}[Lemma $5.11$ in \cite{Sa13}] \label{lem:Pconstr} Let
  $\cN \in M_n(x)|\kappa$ be an $n$-suitable premouse such that for
  some cardinal cutpoint $\eta < \delta^\cN$ of $M_n(x)$, we have that
  $\cN | \delta^\cN \in \cP(M_n(x) | \eta^+) \cap M_n(x) | (\eta^+ + \omega)$ and
  $M_n(x) | \eta$ is generic over $\cN$ for the $\delta^\cN$-generator
  version of the extender algebra at $\delta^\cN$. Then
  $\cN \in \cG$ and
  \[ \cP^{M_n(x)}(\cN | \delta^\cN) = M_n(\cN | \delta^\cN). \] In particular, $M_n(\cN|\delta^\cN)[M_n(x)|\eta] = M_n(x)$.
\end{lemma}

Using pseudo-genericity iterations (see Lemma \ref{lem:pseudogenit})
we can obtain the following corollary.

\begin{cor}\label{cor:itintoG}
  Let $\cN$ be a short tree iterable $n$-suitable premouse such that
  $\cN \in M_n(x) | \kappa$. Then there is a correctly guided finite
  stack on $\cN$ with last model $\cM$ such that $\cM \in \cG$
  and $\cP^{M_n(x)}(\cM | \delta^\cM) = M_n(\cM | \delta^\cM)$.
\end{cor}

Now the following definition of $s$-iterability agrees with the
previous one given in Definition \ref{def:siterable} for $n$-suitable
premice in $\cG$.

\begin{definition}\label{def:siterableinternal}
  For $\cN \in \cG$, $s \in [\Ord]^{<\omega}$, and $k < \omega$
  let
  \begin{align*}
    T_{s,k}^{\cN,*} = \{ (t, \ulcorner\phi\urcorner) \in 
                        [((\delta^\cN)^{+k})^\cN]^{<\omega} & \times
                                                              \omega \st \phi \text{ is a }
                        \Sigma_1\text{-formula and } \\
                      & \cP^{M_n(x)}(\cN | \delta^\cN) \vDash \phi[t,s]
                        \}. 
  \end{align*}
  Then we say for $\cN \in \cG$ and $s \in [\Ord]^{<\omega}$ that
  $M_n(x) \vDash$ ``$\cN$ is $s$-iterable below $\kappa$'' iff for every $\Col(\omega, {<}\kappa)$-generic $G$ over $M_n(x)$ and every correctly guided finite stack $\vec{\cT} = (\cT_i, \cN_i \st i \leq m) \in \HC^{M_n(x)[G]}$ on
  $\cN$ with last model $\cM \in \cG$, there is a
  sequence of branches $\vec{b} = (b_i \st i \leq m) \in M_n(x)[G]$
  and a sequence of embeddings $(\pi_i \st i \leq m)$ satisfying
  $(i)-(iii)$ in Definition \ref{def:siterable} such that if we let
  $\pi_{\vec{\cT},\vec{b}} = \pi_m \circ \pi_{m-1} \circ \dots \circ
  \pi_0$, then for every $k<\omega$,
  \[ \pi_{\vec{\cT},\vec{b}}(T_{s,k}^{\cN,*}) = T_{s,k}^{\cM,*}. \]
\end{definition}

In addition, we define $M_n(x) \vDash$ ``$\cN$ is strongly
$s$-iterable below $\kappa$'' analogous to Definition
\ref{def:stronglysiterable} for all $\Col(\omega,{<}\kappa)$-generic $G$ and stacks $\vec\cT, \vec\cU \in M_n(x)[G]$. For $\cN \in \cG$,
$s \in [\Ord]^{<\omega}$, and $k < \omega$, we have
$T_{s,k}^{\cN,*} = T_{s,k}^{\cN}$, so we will omit the $*$ for
$\cN \in \cG$. Using this, $\gamma_s^{\cN}$ and $H_s^{\cN}$ are
defined as before. Then we can define the internal direct limit system
as follows.

\begin{definition}
  Let
  \begin{align*}
     \cI = \{ (\cN,s) & \st \cN \in \cG, s \in
    [\Ord]^{<\omega}, \text{ and } \\ & M_n(x) \vDash \text{``} \,\cN \text{
      is strongly } s\text{-iterable below } \kappa\text{''} \}
  \end{align*}
  and
  \[ \cF = \{ H_s^{\cN} \st (\cN,s) \in \cI \}. \]
  Moreover, for $(\cN,s), (\cM,t) \in \cI$ we let
  $(\cN,s) \leq (\cM,t)$ iff there is a correctly guided finite
  stack on $\cN$ with last model $\cM$ and $s \subseteq t$. In this
  case we let as before
  $\pi_{(\cN,s), (\cM,t)}:H_s^{\cN} \rightarrow H_t^{\cM}$ denote the
  canonical corresponding embedding.
\end{definition}

For clarity, we sometimes write $\leq_\cI$ for $\leq$. Similar as before we have that for every $\cN \in \cG$ and $s \in [\Ord]^{<\omega}$ there is a normal correct iterate $\cM$ of $\cN$ such that $(\cM,s) \in \cI$. Using the fact that $\kappa$ is inaccessible and a limit of cutpoints in $M_n(x)$ we can obtain the following lemma.

\begin{lemma}[Lemma $5.14$ in \cite{Sa13}]
  $\leq$ is directed.
\end{lemma}

Therefore we can again define the direct limit.

\begin{definition}
  Let $\dlm$ be the direct limit of $(\cF, \leq)$ under
  the embeddings $\pi_{(\cN,s),(\cM,t)}$. Moreover, let
  $\deltadlm = \delta^{\dlm}$ be the Woodin cardinal in $\dlm$ and
  $\pi_{(\cN,s), \infty} : H_s^{\cN} \rightarrow \dlm$ be the direct
  limit embedding for all $(\cN, s) \in \cI$.
\end{definition}

An argument similar to the one for Lemma \ref{lem:dlmequal} shows that
this direct limit is well-founded as well. As we will use ideas from
this proof in the next section, we will give some details here. We
again first define another direct limit system which consists of
iterates of $M_{n+1}^-$ and then show that its direct limit $\dlm^+$
is equal to $\dlm$.

\begin{definition}
  Let
  \begin{align*}
    \cF^+ = \{ \cQ \in \cG \st \cQ & \text{ is the last
      model of a correctly guided} \\ & \text{finite stack on } M_{n+1}^- \text{
      via } \Sigma_{M_{n+1}^-} \}.
  \end{align*}
  Moreover, let $\cP \leq^+ \cQ$ for $\cP, \cQ \in \cF^+$
  iff there is a correctly guided finite stack on $\cP$ according to
  the tail strategy $\Sigma_{\cP}$ with last model $\cQ$. In this case
  we let $i_{\cP, \cQ}: \cP \rightarrow \cQ$ denote the corresponding
  iteration embedding.
\end{definition}

Then $\leq^+$ on $\cF^+$ is directed, so we can define
the direct limit.

\begin{definition}
  Let $\dlm^+$ be the direct limit of $(\cF^+, \leq^+)$
  under the embeddings $i_{\cP, \cQ}$. Moreover, let
  $i_{\cQ,\infty}: \cQ \rightarrow \dlm^+$ denote the direct limit
  embedding for all $\cQ \in \cF^+$.
\end{definition}

Then it is easy to see that $\dlm^+$ is well-founded as $\cF^+$
only consists of iterates of $M_{n+1}^-$ according to the canonical
iteration strategy $\Sigma_{M_{n+1}^-}$.

\begin{lemma}[Lemma $5.15$ in \cite{Sa13}] \label{lem:dlmeqdlmplus}
  $\dlm^+ = \dlm$ and hence $\dlm$ is well-founded.
\end{lemma}

\begin{proof}
  We construct a sequence $(\cQ_i \st i < \omega)$ of iterates of
  $M_{n+1}^-$ such that $\cQ_i \in \cF^+$ for every
  $i < \omega$ and $(\cQ_i \st i < \omega)$ is cofinal in
  $\cG$, i.e., for every $\cN \in \cG$ there is an
  $i<\omega$ such that $\cQ_i$ is the last model of a correctly guided
  finite stack on $\cN$.

  In $V$, fix some sequence $(\xi_i \st i < \omega)$ of ordinals cofinal in $\kappa$. We define $(\cQ_i \st i < \omega)$ together with a strictly
  increasing sequence $(\eta_i \st i < \omega)$ of cardinal cutpoints
  of $M_n(x)|\kappa$ by induction on $i < \omega$. So let
  $\cQ_0 = M_{n+1}^-$ and let $\eta_0 < \kappa$ be a cardinal cutpoint
  of $M_n(x)$. Moreover assume that we already constructed
  $(\cQ_i \st i \leq j)$ and $(\eta_i \st i \leq j)$ with the above
  mentioned properties such that in addition
  $(\cQ_i \st i \leq j) \in M_n(x) | \eta_j$. Let $\cQ_{j+1}^*$ be the
  result of simultaneously pseudo-comparing (in the sense of Lemma
  \ref{lem:pseudocomp}) all $n$-suitable premice $\cM$ such that
  $\cM \in \cG \cap M_n(x) | \eta_j$. Then in particular
  $\cQ_{j+1}^*$ is a normal iterate of $\cQ_{j}$ according to the
  canonical tail iteration strategy $\Sigma_{\cQ_j}$, but
  $\cQ_{j+1}^*$ might not be in $\cG$. Let $\nu$ be a cardinal
  cutpoint of $M_n(x)$ such that $\eta_j < \nu < \kappa$ and
  $\cQ_{j+1}^* \in M_n(x) | \nu$. Note that such a $\nu$ exists as
  $\kappa$ is inaccessible and a limit of cardinal cutpoints in
  $M_n(x)$. Let $\cQ_{j+1}$ be the normal iterate of $\cQ_{j+1}^*$
  according to the canonical tail strategy $\Sigma_{\cQ_{j+1}^*}$ of
  $\Sigma_{\cQ_j}$ obtained by Woodin's genericity iteration such that
  $M_n(x) | \nu$ is generic over $\cQ_{j+1}$ for the
  $\delta^{\cQ_{j+1}}$-generator version of the extender algebra (see
  for example Section $4.1$ in \cite{Fa}). Then
  $\cQ_{j+1} \in \cG$ is as desired. Finally choose
  $\eta_{j+1} < \kappa$ such that $\eta_{j+1} > \max(\eta_j, \xi_j)$, $\eta_{j+1}$ is a cardinal cutpoint in $M_n(x)$ and $(\cQ_i \st i \leq j+1) \in M_n(x) | \eta_{j+1}$.
  
  Now we define an embedding $\sigma : \dlm \rightarrow \dlm^+$ as
  follows. Let $x \in \dlm$. Since $(\cQ_i \st i < \omega)$ is cofinal
  in $\cG$, there are $i,m<\omega$ such that
  $(\cQ_i,s_m) \in \cI$ and
  $x = \pi_{(\cQ_i,s_m),\infty}(\bar{x})$ for some
  $\bar{x} \in H_{s_m}^{\cQ_i} \subseteq \cQ_i$. Then we let
  $\sigma(x) = i_{\cQ_i,\infty}(\bar{x})$.

  It follows as in the proof of Lemma $5.10$ in \cite{Sa13} that the
  definition of $\sigma$ does not depend on the choice of
  $i,m < \omega$ and in fact $\sigma = \id$.
\end{proof}

Moreover, it is possible to compute $\deltadlm$.

\begin{lemma}[Lemma $5.16$ in \cite{Sa13}]
  $\deltadlm = (\kappa^+)^{M_n(x)}$.
\end{lemma}

\subsection*{Direct limit systems in $\HODMn$}

Finally, we will argue that $\dlm \in \HODMn$ by first defining direct
limit systems in various premice $M(y)$ satisfying certain properties
definable in $M_n(x)[g]$ and then showing that the direct limits
$\dlm^{M(y)}$ are equal to $\dlm$. A similar approach but in a
completely different setting can be found in \cite{SaSch18}.

In what follows, we will let $(K(z))^N$ denote the core model
constructed above a real $z$ inside some $n$-small model $N$ with $n$
Woodin cardinals in the sense of \cite{Sch06}, i.e., the core model
$K(z)$ is constructed between consecutive Woodin cardinals. Lemma
$1.1$ in \cite{Sch06} (due to John Steel) implies that
$(K(x))^{M_n(x)} = M_n(x)$. We will use this fact and consider more
arbitrary premice with this property in what follows. We
state the following definitions in $V$, but we will later apply them inside
$M_n(x)[g]$.

\begin{definition}\label{def:predlmsuit}
  Let $y \in \BS \cap M_n(x)[g]$. Then we say $y$ is
  \emph{pre-dlm-suitable} iff there is a proper class $y$-premouse
  $M(y)$ satisfying the following properties.
  \begin{enumerate}[$(i)$]
  \item $M(y)$ is $n$-small and has $n$ Woodin cardinals, 
  \item the least inaccessible cardinal in $M(y)$ is $\kappa$,
  \item $M(y) = (K(y))^{M(y)}$, and
  \item there is a $\Col(\omega, {<}\kappa)$-generic $h$ over $M(y)$
    such that \[ M(y)[h] = M_n(x)[g]. \] 
  \end{enumerate}
  We also call such a $y$-premouse $M(y)$ \emph{pre-dlm-suitable} and say
  that \emph{$M(y)$ witnesses that $y$ is pre-dlm-suitable}.
\end{definition}

Using this, we can define a version of the direct limit system
$\cF$ inside arbitrary pre-dlm-suitable $y$-premice $M(y)$.

\begin{definition}
  Let $y \in \BS$ be pre-dlm-suitable as witnessed by $M(y)$. Then we let 
   \begin{align*}
    \cG^{M(y)} & = \;  \{ \cN \in M(y)|\kappa \st \cN \text{ is }
                            n\text{-suitable and } M(y) \vDash \text{``for
                      some cardinal}\\
                 & \; \text{cutpoint } \eta,  
                   \delta^\cN = \eta^+, \cN | \delta^\cN \in \cP(M(y) | \eta^+) \cap M(y) | (\eta^+ + \omega), \\
                &  \;  \text{and } M(y) | \eta \text{ is generic
                   over } \cN \text{ for the } \delta^\cN\text{-generator version of} \\ & \; \; \; \; \; \; \text{the extender algebra at } \delta^\cN \text{''} \}.
  \end{align*}
\end{definition}

Analogous as before, we can now define when for an $n$-suitable premouse $\cN$, $M(y) \vDash \text{``} \cN$ is strongly $s$-iterable below $\kappa$'' by referring to $\cP^{M(y)}(\cN|\delta^\cN)$ in the definition of $(T_{s,k}^{\cN,*})^{M(y)}$. Let $\gamma_s^{\cN,M(y)}$ and $H_s^{\cN,M(y)}$ be defined analogous to $\gamma_s^{\cN}$ and $H_s^{\cN}$ inside $M(y)$ using $(T_{s,k}^{\cN,*})^{M(y)}$. For $M(y) = M_n(x)$ and $\cN \in \cG$ this agrees with our previous definition of strong $s$-iterability.

\begin{definition}
  Let $y \in \BS$ be pre-dlm-suitable as witnessed by $M(y)$. Then we let
  \begin{align*}
    \cI^{M(y)} & = \{ (\cN,s) \st \cN \in \cG^{M(y)}, s \in
                                    [\Ord]^{<\omega}, \text{ and} \\ & M(y)
                                                               \vDash
                                                               \text{``}
                                                               \,\cN
                                                               \text{
                                                               is
                                                               strongly
                                                               }
                                                               s\text{-iterable
                                                               below }
                                                               \kappa\text{''}
                                                               \}  
  \end{align*}
  and
  \[ \cF^{M(y)} = \{ H_s^{\cN,M(y)} \st (\cN,s) \in
    \cI^{M(y)} \}. \] Moreover, for
  $(\cN,s), (\cM,t) \in \cI^{M(y)}$ we let
  $(\cN,s) \leq_{\cI^{M(y)}} (\cM,t)$ iff there is a
  correctly guided finite stack on $\cN$ with last model $\cM$ and
  $s \subseteq t$. In this case we let
  $\pi^{M(y)}_{(\cN,s), (\cM,t)}:H_s^{\cN,M(y)} \rightarrow
  H_t^{\cM,M(y)}$ denote the canonical corresponding embedding. Finally, let $\dlm^{M(y)}$ denote the direct limit of
  $(\cF^{M(y)}, \leq_{\cI^{M(y)}})$ under these
  embeddings.
\end{definition}

We will now strengthen this and define when a real $y \in \BS$ (or a
$y$-premouse $M(y)$) is dlm-suitable.

\begin{definition}\label{def:dlmsuit}
  Let $y \in \BS \cap M_n(x)[g]$ be pre-dlm-suitable as witnessed by
  some $y$-premouse $M(y)$. We say that $y$ is \emph{dlm-suitable}
  (witnessed by $M(y)$) iff
  \begin{enumerate}[$(i)$]
  \item for every $s \in [\Ord]^{<\omega}$ there is a premouse $\cN$
    such that $(\cN,s) \in \cI^{M(y)}$, and
  \item for every $\cN \in \cG^{M(y)}$,
    \[\cP^{M(y)}(\cN|\delta^{\cN}) = K^{M_n(x)[g]}(\cN|\delta^{\cN}).\]
  \end{enumerate}
\end{definition}

\begin{lemma}
  $M_n(x)$ witnesses that $x$ is dlm-suitable.
\end{lemma}
\begin{proof}
  The fact that $M_n(x)$ satisfies $(i)$ follows from Lemma
  \ref{lem:stronglysiterable} and Corollary \ref{cor:itintoG}, so we
  only have to show $(ii)$.  Let $\cN \in \cG$. Then
  $\cP^{M_n(x)}(\cN|\delta^\cN) = M_n(\cN|\delta^\cN)$ by Lemma
  \ref{lem:Pconstr}. Moreover, there is some $G$ generic over the
  result of the $\cP$-construction $\cP^{M_n(x)}(\cN|\delta^\cN)$ for
  the $\delta^\cN$-generator version of the extender algebra at
  $\delta^\cN$ with $\cP^{M_n(x)}(\cN|\delta^\cN)[G] = M_n(x)$. That
  means
  \[ M_n(\cN|\delta^\cN)[G] = M_n(x). \] Now,
  \begin{align*}
    K^{M_n(x)[g]}(\cN|\delta^\cN) & = 
    K^{M_n(\cN|\delta^\cN)[G][g]}(\cN|\delta^\cN) =
    K^{M_n(\cN|\delta^\cN)}(\cN|\delta^\cN) \\ & = M_n(\cN|\delta^\cN) =
    \cP^{M_n(x)}(\cN|\delta^\cN), 
  \end{align*}
  by generic absoluteness of the core model and Lemma 1.1 in
  \cite{Sch06} (due to Steel).
\end{proof}

Condition $(ii)$ in Definition \ref{def:dlmsuit} will ensure that for
any dlm-suitable $y$-premouse $M(y)$ and
$(\cN,s), (\cM,t) \in \cI \cap \cI^{M(y)}$ with
$(\cN,s) \leq_{\cI} (\cM,t)$ and
$(\cN,s) \leq_{\cI^{M(y)}} (\cM,t)$, the induced embeddings
$\pi_{(\cN,s), (\cM,t)}$ and $\pi^{M(y)}_{(\cN,s), (\cM,t)}$
agree. Hence we can show in the following lemma that the direct limit
$\dlm^{M(y)}$ defined inside some dlm-suitable $M(y)$ will in fact be
the same as the direct limit $\dlm$ defined inside $M_n(x)$.

\begin{lemma}\label{lem:dlmM(y)eq}
  Let $y \in \BS$ be dlm-suitable as witnessed by $M(y)$. Then
  $\cF$ and $\cF^{M(y)}$ have cofinally many points in
  common and $\dlm = \dlm^{M(y)}$.
\end{lemma}

\begin{proof}
  Let $h$ be $\Col(\omega, {<}\kappa)$-generic over $M(y)$ such that
  $M(y)[h] = M_n(x)[g]$. Let $(\cN,s) \in \cI$ and
  $(\cN^\prime,s^\prime) \in \cI^{M(y)}$. We aim to show that
  there is some $(\cM, t) \in \cI \cap \cI^{M(y)}$ such
  that $(\cN,s) \leq_{\cI} (\cM,t)$ and
  $(\cN^\prime, s^\prime) \leq_{\cI^{M(y)}} (\cM,t)$. As
  condition $(ii)$ in Definition \ref{def:dlmsuit} yields that the
  embeddings associated to $\cF$ and $\cF^{M(y)}$ agree,
  this suffices to show that $\dlm = \dlm^{M(y)}$.

  Let $t = s \cup s^\prime$. By assumption, there is a $t$-iterable
  premouse $\cR$ in $M_n(x)$ and a $t$-iterable premouse $\cR^\prime$
  in $M(y)$. Therefore we can assume that $\cN$ and $\cN^\prime$ are
  both $t$-iterable in $M_n(x)$ and $M(y)$ respectively as we can
  replace them by the result of their coiteration with $\cR$ and
  $\cR^\prime$ respectively.

  By the choice of $M(y)$ and generic absoluteness of the core model
  we have
  \begin{align}
    \label{eq:1}
    M(y) & = (K(y))^{M(y)} = (K(y))^{M(y)[h]} \\ & = (K(y))^{M_n(x)[g]} =
    (K(y))^{M_n(x)[g \upharpoonright \xi]}, \nonumber
  \end{align}
  where $\xi < \kappa$ is
  such that $y \in M_n(x)[g \upharpoonright \xi]$. Analogously, using
  Lemma $1.1$ in \cite{Sch06} due to Steel and generic absoluteness of
  the core model again,
  \begin{align}
    \label{eq:2}
    M_n(x) & = (K(x))^{M_n(x)} = (K(x))^{M_n(x)[g]} \\ & = (K(x))^{M(y)[h]}
    = (K(x))^{M(y)[h \upharpoonright \xi^\prime]}, \nonumber
  \end{align}
  where $\xi^\prime < \kappa$ is such that
  $x \in M(y)[h \upharpoonright \xi^\prime]$. Now we can obtain the following claim. 

  \begin{claim}\label{cl:ctpts}
    $M(y)$ and $M_n(x)$ have cofinally many common 
    cardinal cutpoints below $\kappa$.
  \end{claim}
  \begin{proof}
    As $M(y) = K(y)^{M_n(x)[g \upharpoonright\xi]}$ is an inner model of $M_n(x)[g \upharpoonright \xi]$, every cardinal above $\xi$ in $M_n(x)$ is a cardinal in $M(y)$, so it is easy to find cofinally many common cardinals of $M(y)$ and $M_n(x)$ below $\kappa$. Recall that $\kappa$ is the least inaccessible cardinal of both $M(y)$ and $M_n(x)$. Therefore, every common cardinal of $M(y)$ and $M_n(x)$ below $\kappa$ is also a cutpoint of both models.
  \end{proof}

Moreover, Equations \eqref{eq:1} and \eqref{eq:2} yield
\[M(y) \subseteq M_n(x)[g \upharpoonright \xi] \subseteq M(y)[h
  \upharpoonright \zeta],\] where $\xi^\prime < \zeta < \kappa$ is
such that $g \upharpoonright \xi \in M(y)[h \upharpoonright
\zeta]$. By the intermediate model theorem (see for example Lemma
$15.43$ in \cite{Je03}) this implies that
$M_n(x)[g \upharpoonright \xi]$ is a generic extension of $M(y)$ for a
forcing of size less than $\kappa$.\footnote{I.e. $M(y)$ is a ground
  of $M_n(x)[g \upharpoonright \xi]$. See for example \cite{FHR15} or
  \cite{Us17} for an introduction to the theory of grounds.} Since
$M_n(x)[g \upharpoonright \xi]$ is a generic extension of $M_n(x)$ for
a forcing of size less than $\kappa$ as well, this implies by Theorem
$1.3$ in \cite{Us17} that there is some common inner model
$W \subseteq M_n(x) \cap M(y)$ such that
$M_n(x)[g \upharpoonright \xi]$ is a generic extension of $W$ for a
forcing of size less than $\kappa$.

  As every generic extension via a forcing of size less than $\kappa$
  can be absorbed by the collapse of some ordinal $\beta < \kappa$,
  this yields that we can fix some ordinal $\beta < \kappa$ and some
  $\Col(\omega,\beta)$-generic $b \in M_n(x)[g]$ over $W$ such that
  $x,y,\cN,\cN^\prime \in W[b]$. Then $M_n(x)$ and $M(y)$ exist in
  $W[b]$ as definable subclasses because
  \[(K(x))^{W[b]} = (K(x))^{M_n(x)[g \upharpoonright \xi]} =
    (K(x))^{M_n(x)} = M_n(x)\] and similarly
  \[(K(y))^{W[b]} = (K(y))^{M_n(x)[g \upharpoonright \xi]} =
    (K(y))^{M_n(x)[g]} = (K(y))^{M(y)[h]}= M(y)\] by generic
  absoluteness of the core model again. Let
  $\dot{x},\dot{y},\dot{\cN}$ and $\dot{\cN^\prime}$ be
  $\Col(\omega,\beta)$-names for $x,y,\cN$ and $\cN^\prime$ in
  $W$. Moreover, let $p \in \Col(\omega,\beta)$ force all properties
  we need about $\dot{x},\dot{y},\dot{\cN}$ and
  $\dot{\cN^\prime}$. For $q \leq_{\Col(\omega,\beta)} p$ let $b_q$ be
  the $\Col(\omega,\beta)$-generic filter over $W$ such that
  $\bigcup b_q$ agrees with $q$ on $\dom(q)$ and with $\bigcup b$
  everywhere else.

  Now we construct $(\cM, t) \in \cI \cap
  \cI^{M(y)}$. Let $\eta < \kappa$ be a cardinal
  cutpoint of both $M(y)$ and $M_n(x)$ such that
  $\xi, \xi^\prime < \eta$, which exists by Claim \ref{cl:ctpts}. Then
  in fact $(\eta^+)^{M_n(x)} = (\eta^+)^{M(y)}$ as by Equations
  $(\ref{eq:1})$ and $(\ref{eq:2})$ at the beginning of the proof
  \[ (\eta^+)^{M(y)} \leq (\eta^+)^{M_n(x)[g \upharpoonright \xi]} =
    (\eta^+)^{M_n(x)} \leq (\eta^+)^{M(y)[h \upharpoonright
      \xi^\prime]} = (\eta^+)^{M(y)}. \] By the same argument,
  $(\eta^+)^{K(\dot{x}^{b_q})} = (\eta^+)^{K(\dot{y}^{b_q})}$ for all
  $q \leq_{\Col(\omega,\beta)} p$.

  Work in $W[b]$. Using Lemmas \ref{lem:pseudocomp} and
  \ref{lem:pseudogenit}, we obtain an inner model $\cM$ by
  pseudo-comparing all $(\dot{\cN})^{b_q}$ and
  $(\dot{\cN^\prime})^{b_q}$ for $q \leq_{\Col(\omega,\beta)} p$ and
  simultaneously pseudo-genericity iterating such that
  $K(\dot{x}^{b_q}) | \eta$ and $K(\dot{y}^{b_q}) | \eta$ are generic
  over $\cM$ and
  $\delta^\cM = (\eta^+)^{K(\dot{x}^{b_q})} =
  (\eta^+)^{K(\dot{y}^{b_q})}$. Since $\cM$ is definable in $W[b]$
  from $\{b_q \st q \leq_{\Col(\omega,\beta)} p \}$ and parameters
  from $W$, we have that in fact
  $\cM \in W \subseteq M_n(x) \cap M(y)$, as $\cM$ does not depend on
  the choice of the generic $b$. Moreover, $\cM$ is a correct iterate
  of $\cN$ in $M_n(x)$ and a correct iterate of $\cN^\prime$ in
  $M(y)$.

  As argued above, we can assume that $\cN$ and $\cN^\prime$ are
  $t$-iterable in $M_n(x)$ and $M(y)$ respectively for
  $t = s \cup s^\prime$. Therefore $\cM$ is $t$-iterable in both,
  $M_n(x)$ and $M(y)$. Hence,
  $(\cM,t) \in \cI \cap \cI^{M(y)}$,
  $(\cN,s) \leq_{\cI} (\cM,t)$, and
  $(\cN^\prime, s^\prime) \leq_{\cI^{M(y)}} (\cM,t)$, as
  desired.
\end{proof}

This yields that $\dlm \in \HODMn$.

\section{$\HOD$ below $\deltadlm$}\label{sec:dmr}

In this section we will show that $\HODMn$ and $\dlm$ agree up to
$\deltadlm$ by generalizing the arguments in Section $3.4$ in
\cite{StW16}. We show a version of Woodin's derived model resemblance for our setting. For this, we do not need to talk about generic extensions such as $M(y)[h]$ and work with $M(y)$ directly instead.

We start with expanding our direct limit $\dlm$ to the proper class premouse
$\dlmhat = M_n(\dlm | \deltadlm)$ and define a direct limit system
$\hat{\cF}$ of expansions of the elements of $\cF$ that
converges to $\dlmhat$. For an $n$-suitable premouse $\cN$ and $s \in [\Ord]^{<\omega}$ with
$max(s)$ a uniform indiscernible above the Woodin cardinals in
$M_n(\cN|\delta^{\cN})$, we let $\hat{\cN} = M_n(\cN | \delta^{\cN})$
be the proper class expansion of $\cN$, $s^- = s \setminus max(s)$,
 \[ \gamma_s^{\hat{\cN}} = \sup(Hull_1^{\hat{\cN} |
     max(s)}(s^-) \cap \delta^\cN), \] and
 \[ H_s^{\hat{\cN}} = Hull_1^{\hat{\cN} | max(s)}(\gamma_s^{\hat{\cN}}
   \cup s^-). \] Now we let
 \[ \hat{\cF} = \{ H_s^{\hat{\cN}} \st (\cN,s) \in \cI
   \} \] and for $(\cN,s), (\cM,t) \in \cI$ with
 $(\cN,s) \leq (\cM,t)$ we let
 $\hat{\pi}_{(\cN,s),(\cM,t)} : H_s^{\hat{\cN}} \rightarrow
 H_t^{\hat{\cM}}$ denote the canonical corresponding embedding extending $\pi_{(\cN,s^-),(\cM,t^-)}$.

 Finally, let $\dlmhat$ be the direct limit of
 $(\hat{\cF}, \leq)$ under the embeddings
 $\hat{\pi}_{(\cN,s),(\cM,t)}$ and let $\hat{\pi}_{(\cN,s),\infty} : H_s^{\hat{\cN}} \rightarrow \dlmhat$
 for $(\cN,s) \in \cI$ denote the direct limit embedding. Then it is easy to see that
 $M_n(\dlm | \deltadlm) = \dlmhat$ as we can define a similar direct limit system $\hat{\cF}^+$ of premice expanding the $n$-suitable premice in $\cF^+$.

Choose for
 any ordinal $\alpha$ an arbitrary $(\cN,s) \in \cI$ such that
 $\alpha \in s \setminus max(s)$ and let
 $\alpha^* = \hat{\pi}_{(\cN,s),\infty}(\alpha)$. Note that the value
 of $\alpha^*$ does not depend on the choice of $(\cN,s)$. We also let
 $t^* = \{ \alpha^* \st \alpha \in t \}$ for $t \in [\Ord]^{<\omega}$.

 \begin{lemma}\label{lem:dmr}
   Let $\cN$ be an $n$-suitable premouse and $s \in [\Ord]^{<\omega}$
   such that $(\cN, s) \in \cI$. Let
   $\bar{\xi} < \gamma_s^{\cN}$,
   $\xi = \pi_{(\cN,s),\infty}(\bar{\xi})$ and
   $t \in [\Ord]^{<\omega}$. Moreover, let $\varphi(v_0,v_1)$ be a
   formula in the language of premice, i.e., we allow the extender sequence as a predicate. Then the following are equivalent.
  \begin{enumerate}[$(a)$]
  \item $M_n(\dlm | \deltadlm) \vDash \varphi(\xi, t^*)$,
  \item in $M_n(x)[g]$, there is some dlm-suitable $y \in \BS$
    witnessed by $M(y)$ with $(\cN,s) \in \cI^{M(y)}$ and a correctly guided
    finite stack on $\cN$ with last model $\cM \in M(y)$ such that
    whenever $\cR \in \cG^{M(y)}$ is the last model of a
    correctly guided finite stack on $\cM$, then
    $\cP^{M(y)}(\cR|\delta^\cR) \vDash \varphi(\pi^{M(y)}_{(\cN,s),(\cR,s)}(\bar{\xi}), t)$. 
  \end{enumerate}
\end{lemma}

\begin{proof}
  To prove that $(a)$ implies $(b)$ we assume toward a contradiction
  that $(b)$ is false. So in $M_n(x)[g]$ for all dlm-suitable $y \in \BS$ and $M(y)$
  witnessing this with $(\cN, s) \in \cI^{M(y)}$ and all correctly guided finite
  stacks on $\cN$ with last model $\cM \in M(y)$, there is a correctly
  guided finite stack on $\cM$ with last model
  $\cR \in \cG^{M(y)}$ such that
  $\cP^{M(y)}(\cR|\delta^\cR) \vDash \neg \varphi(\pi^{M(y)}_{(\cN,s),(\cR,s)}(\bar{\xi}), t)$.

  We can assume without loss of generality that $\cN \in M_n(x)$ is
  the last model of a correctly guided finite stack on $M_{n+1}^-$ via
  the canonical iteration strategy $\Sigma_{M_{n+1}^-}$ and strongly
  $s$-iterable below $\kappa$. Moreover, we can assume that $max(s)$ is a uniform indiscernible. If this is not already the case, we replace $\cN$ by a
  pseudo-iterate of the result of the pseudo-comparison of $\cN$ with
  $M_{n+1}^-$ using Lemma \ref{lem:pseudocomp} and Corollary
  \ref{cor:itintoG}.

  \begin{claim}
    There are $n$-suitable premice $\cN_k \in \cF^+$ for
    $k < \omega$ which are cofinal in $\cF^+$ such that
    $\cN_0 = \cN$ and for all $k<\omega$,
    \[ M_n(\cN_k|\delta^{\cN_k}) \vDash \neg\varphi(\bar{\xi}_k,t), \] where
    $\bar{\xi}_k = i_{\cN_0,\cN_k}(\bar{\xi})$ is the image of
    $\bar{\xi}$ under the iteration map induced by
    $\Sigma_{M_{n+1}^-}$.
  \end{claim}
  \begin{proof} Let $(\cQ_i \st i < \omega)$ be an enumeration of
    $\cF^+$ and $\cN_0 = \cN$. Then we construct $\cN_{k+1}$
    inductively. So assume that we already constructed $\cN_k$ and
    pseudo-coiterate $\cN_k$ with $\cQ_k$ to some model $\cN_k^*$ (see
    Lemma \ref{lem:pseudocomp}). By assumption $(b)$ is false, so let
    $\cR$ be a counterexample witnessing this for $\cN_k^*$ and the
    dlm-suitable premouse $M_n(x)$. That means $\cR \in \cG$ is
    the last model of a correctly guided finite stack on $\cN_k^*$
    such that
    $M_n(\cR|\delta^\cR) = \cP^{M_n(x)}(\cR|\delta^\cR) \vDash \neg \varphi(i_{\cN,\cR}(\bar{\xi}),t)$ as
    $i_{\cN,\cR} \upharpoonright H_s^{\cN} =
    \pi_{(\cN,s),(\cR,s)}$. But $\cR \in \cF^+$ since
    $\cR \in \cG$ and it is a correct iterate of $\cQ_k$. Thus
    we can let $\cN_{k+1} = \cR$.
  \end{proof}

  Since $(\cN_k \st k < \omega)$ is cofinal in
  $\cF^+$, it follows that the direct limit of
  $(\cN_k, i_{\cN_k,\cN_l} \st k<l<\omega)$ is equal to $\dlm^+$. Let
  $\hat{\cN_k} = M_n(\cN_k | \delta^{\cN_k})$ and let
  $\hat{i}_{\hat{\cN}_k,\infty} : \hat{\cN_k} \rightarrow M_n(\dlm^+ |
  \delta^{\dlm^+}) = \dlmhat^+$ be the corresponding extension of the
  direct limit map $i_{\cN_k, \infty}$. Then we have for all sufficiently large $k$ that
  \[ M_n(\dlm^+ | \delta^{\dlm^+}) \vDash \neg\varphi(i_{\cN_k,
      \infty}(\bar{\xi}_k),
    \hat{i}_{\hat{\cN}_k,\infty}[t]). \] Since we assumed
  that $\cN$ is strongly $s$-iterable below $\kappa$ and
  $\bar{\xi} < \gamma_s^{\cN}$, it follows that
  $i_{\cN_k, \infty}(\bar{\xi}_k) = i_{\cN,\infty}(\bar{\xi}) =
  \pi_{(\cN,s),\infty}(\bar{\xi}) = \xi$ as
  $\bar{\xi}_k = i_{\cN,\cN_k}(\bar{\xi})$.

  Let $k<\omega$ be large enough such that
  $(\cN_k, s \cup t) \in \cI$ and
  $\hat{i}_{\hat{\cN}_l,\hat{\cN}_{l+1}}(s \cup t) = s \cup t$ for all $l \geq
  k$. Such a $k$ exists by a so-called bad sequence argument similar
  to the one in the proof of Lemma $5.8$ in \cite{Sa13}.



 Consider the map
  \[ \hat{\sigma}: M_n(\dlm|\deltadlm) = \dlmhat \rightarrow \dlmhat^+ = 
    M_n(\dlm^+|\delta^{\dlm^+}) \] which is the canonical extension of
  the map $\sigma: \dlm \rightarrow \dlm^+$ defined in the proof of
  Lemma \ref{lem:dlmeqdlmplus}. That means, for $x \in \dlmhat$, say
  $x = \hat{\pi}_{(\cN_l,r),\infty}(\bar{x})$ for some
  $\bar{x} \in H_r^{\hat{\cN}_l}$, we may assume similar as above that $l \geq k$ is large enough such that $(\cN_l,r) \in \cI$ and $\hat{i}_{\hat{\cN}_j,\hat{\cN}_{j+1}}(r) = r$ for all $j \geq l$. Then we let
  $\hat{\sigma}(x) = \hat{i}_{\hat{\cN}_l,\infty}(\bar{x})$. Now it
  follows from a generalization of the proof of \cite[Claim 2]{Schl} that
  $\hat{\sigma} = \id$. Moreover, we have that
  $\hat{\sigma}[t^*] = \hat{\sigma}(\hat{\pi}_{(\cN_k,s \cup
    t),\infty}[t]) = \hat{i}_{\hat{\cN}_k,\infty}[t]$. Therefore
  pulling back under $\hat{\sigma}$ yields that
  \[ M_n(\dlm | \deltadlm) \vDash \neg\varphi(\xi,t^*). \] This is the desired
  contradiction to $(a)$.

  To show that $(b)$ implies $(a)$ we now assume that $(b)$ is
  true. Let $M(y)$ be the dlm-suitable premouse with $(\cN,s) \in \cI^{M(y)}$
  given by $(b)$. As before we can assume without loss of generality
  that $\cN$ is the last model of a correctly guided finite stack on
  $M_{n+1}^-$ via the canonical iteration strategy
  $\Sigma_{M_{n+1}^-}$, that $\cN$ is strongly $s$-iterable below $\kappa$ with
  respect to branches choosen by $\Sigma_{M_{n+1}^-}$, that $max(s)$ is a uniform indiscernible, and that
  $\cN \in \cG \cap \cG^{M(y)}$ using Lemma \ref{lem:dlmM(y)eq}.

  \begin{claim}
    There are $n$-suitable premice $\cN_k \in \cF^+$ for
    $k < \omega$ which are cofinal in $\cF^+$ such that
    $\cN_0 = \cN$ and for all $k<\omega$,
    \[ M_n(\cN_k|\delta^{\cN_k}) \vDash \varphi(\bar{\xi}_k,t), \] where
    $\bar{\xi}_k = i_{\cN_0,\cN_k}(\bar{\xi})$ is the image of
    $\bar{\xi}$ under the iteration map induced by $\Sigma_{M_{n+1}^-}$.
  \end{claim}

  \begin{proof}
    By the proof of Lemma \ref{lem:dlmM(y)eq}, we can pick a sequence $(\cQ_i \st i < \omega)$ of premice cofinal in $\cF^+$ such that $\cQ_i \in \cF^{M(y)}$ for all $i < \omega$. Let $\cN_0 = \cN$ and construct $\cN_{k+1} \in M(y)$
    inductively. Assume that we already constructed $\cN_k$ and let
    $\cM \in M(y)$ be the last model of a correctly guided finite stack on
    $\cN$ witnessing that $(b)$ is true. Simultaneously
    pseudo-coiterate $\cM$ with $\cN_k$ and $\cQ_k$ to some premouse
    $\cN_k^*$. Using genericity iterations and Lemma
    \ref{lem:dlmM(y)eq}, there is a pseudo-iterate $\cR$ of $\cN_k^*$
    such that $\cR \in \cG \cap \cG^{M(y)}$ (see also
    Corollary \ref{cor:itintoG}). In particular, we have by dlm-suitability of
    $M(y)$ that
    \begin{align*}
      \cP^{M(y)}(\cR|\delta^\cR) & = K^{M_n(x)[g]}(\cR|\delta^\cR) = K^{M_n(\cR|\delta^\cR)[G][g]}(\cR|\delta^\cR) \\ & = K^{M_n(\cR|\delta^\cR)}(\cR|\delta^\cR) = M_n(\cR|\delta^\cR) 
    \end{align*}
    for some $G$ generic over $M_n(\cR|\delta^\cR)$ for the extender algebra and therefore
    $M_n(\cR|\delta^\cR) \vDash \varphi(i_{\cN,\cR}(\bar{\xi}),t)$ using
    $(b)$ as
    $i_{\cN,\cR}(\bar{\xi}) = \pi_{(\cN,s),(\cR,s)}(\bar{\xi}) =
    \pi^{M(y)}_{(\cN,s),(\cR,s)}(\bar{\xi})$. Moreover, $\cR$ is the last model of a
    correctly guided finite stack on $\cQ_k$ and thus
    $\cR \in \cF^+$, so we can let $\cN_{k+1} = \cR$.
  \end{proof}

  As before we can use this claim to obtain that
  \[ M_n(\dlm | \deltadlm) \vDash \varphi(\xi,t^*), \] which proves $(a)$.
\end{proof}

Let $\kappa_\infty$ be the least inaccessible cardinal above $\deltadlm$ in $\dlmhat = M_n(\dlm|\deltadlm)$ and fix some $H$ which is $\Col(\omega, {<}\kappa_\infty)$-generic over $\dlmhat$. Then Lemma \ref{lem:dmr} implies for example that $\dlmhat[H]$ and $M_n(x)[g]$ are elementary equivalent (for formulae in the language of set theory) as for $\cR$ as in the statement of Lemma \ref{lem:dmr}, there is some $\Col(\omega, {<}\kappa^\cR)$-generic $G$, where $\kappa^\cR$ is the least inaccessible cardinal above $\delta^\cR$ in $\cP^{M(y)}(\cR|\delta^\cR)$, such that $\cP^{M(y)}(\cR|\delta^\cR)[G] = M_n(x)[g]$.

We defined a direct limit system $\cF^{M(y)}$ for all dlm-suitable $M(y)$ in $M_n(x)[g]$. Therefore, there is a direct limit system $\cF^{*, M(y)}$ with the same properties for each dlm-suitable $M(y)$ in $\dlmhat[H]$ (adapting the definition of dlm-suitable to $\dlmhat[H]$). It is easy to see that Lemma \ref{lem:dmr} implies that for each $s \in [\Ord]^{<\omega}$ there is a $y_s$ such that $\dlm \in \cG^{M(y_s)}$ and $\dlm$ is strongly $s^*$-iterable in $M(y_s)$ in $\dlmhat[H]$. In fact, the direct limit embedding $(\pi_{(\dlm,s^*),\infty}^{M(y_s)})^{\cF^{*, M(y_s)}}$  in the system $\cF^{*, M(y_s)}$ is independent of the choice of $y_s$ and we can consider
\begin{align*}
  \pi_\infty & = \bigcup \{ (\pi_{(\dlm,s^*),\infty}^{M(y_s)})^{\cF^{*, M(y_s)}}
  \st s \in [\Ord]^{<\omega} \text{ and } y_s \text{ is such that } \\ & \dlm \in \cG^{M(y_s)} \text{ and } \dlm \text{ is strongly } s^*\text{-iterable in } M(y_s) \text{ in } \dlmhat[H] \}.
\end{align*}

\begin{lemma}\label{lem:piinftyid}
  For all $\eta < \deltadlm$ we have that $\pi_\infty(\eta) = \eta^*$.
\end{lemma}

\begin{proof}
  This is again a consequence of Lemma
  \ref{lem:dmr}. Consider the dlm-suitable premouse $M_n(x)$. Let
  $\eta = \pi_{(\cN,s),\infty}(\bar{\eta})$ for some
  $(\cN,s) \in \cI$ and $\bar{\eta} < \gamma_s^{\cN}$ and
  consider the formula
  \begin{align*}
    \varphi(&v_0,v_1,v_2,v_3) = \text{``}1 \forces{}{\Col(\omega,{<}\kappa^*)} \text{for all dlm-suitable } y \text{ with } v_0 \in \cG^{M(y)}, \\ &
    \text{we have } (v_0,v_1) \in \cI^{M(y)}, v_2 <
    \gamma_{v_1}^{v_0, M(y)}, \text{ and } 
    \pi_{(v_0,v_1),\infty}^{M(y)}(v_2) = v_3\text{''},
  \end{align*}
  where $\kappa^*$ refers to the least inaccessible cardinal above the least Woodin cardinal of the current model. Recall that for any dlm-suitable $y$ and $z$ witnessed by $M(y)$ and $M(z)$, for any $(\cN,s), (\cM,t) \in \cI^{M(y)} \cap \cI^{M(z)}$ with $(\cN,s) \leq_{\cI^{M(y)}} (\cM,t)$ and $(\cN,s) \leq_{\cI^{M(z)}} (\cM,t)$ the induced embeddings $\pi_{(\cN,s),(\cM,t)}^{M(y)}$ and $\pi_{(\cN,s),(\cM,t)}^{M(z)}$ agree. Hence, in $M_n(x)[g]$, we have for every $\cR \in \cG$ which is the last model
  of a correctly guided finite stack on $\cN$ that for all dlm-suitable $y$ such that $\cR \in \cG^{M(y)}$, in fact $(\cR,s) \in \cI^{M(y)}$, $\pi_{(\cN,s),(\cR,s)}(\bar{\eta}) < \gamma_{s}^{\cR, M(y)}$, and $\pi_{(\cR,s),\infty}^{M(y)}(\pi_{(\cN,s),(\cR,s)}(\bar{\eta})) = \eta$. Therefore, 
  $\cP^{M(y)}(\cR|\delta^{\cR}) \vDash \varphi(\cR, s,
  \pi_{(\cN,s),(\cR,s)}(\bar{\eta}), \eta)$. So Lemma \ref{lem:dmr} 
  yields that $\dlmhat \vDash \varphi(\dlm, s^*, \eta, \eta^*)$. So in $\dlmhat[H]$, for all dlm-suitable $y$ with $\dlm \in \cG^{M(y)}$, we have $(\dlm,s^*) \in \cI^{M(y)}$ and $(\pi^{M(y)}_{(\dlm,s^*),\infty})^{\cF^*, M(y)}(\eta) = \eta^*$, as desired.
\end{proof}

\begin{thm}\label{thm:HODuptoDelta}
  $V_{\deltadlm}^{\HODMn} = V_{\deltadlm}^{\dlm}$. 
\end{thm}

\begin{proof}
  By the internal definition of $\dlm$ from Lemma \ref{lem:dlmM(y)eq}
  we have that
  $V_{\deltadlm}^{\dlm} \subseteq V_{\deltadlm}^{\HODMn}$. For the
  other inclusion we first show the following claim.
\begin{claim}\label{cl:*}
  $\pi_\infty \upharpoonright \alpha \in \dlmhat$ for all
  $\alpha < \deltadlm$.
\end{claim}

\begin{proof}
  As $\alpha < \deltadlm$, there exists an $s \in [\Ord]^{<\omega}$ and a dlm-suitable $y_s$ such that, in $\dlmhat[H]$,  $\dlm \in \cG^{M(y)}$ and $\alpha < \gamma_{s^*}^{\dlm, M(y_s)}$. We have by definition that
  $\pi_\infty \upharpoonright \alpha =
  (\pi^{M(y_s)}_{(\dlm,s^*),\infty})^{\cF^*, M(y_s)} \upharpoonright
  \alpha$. As $ (\pi^{M(y_s)}_{(\dlm,s^*),\infty})^{\cF^*, M(y_s)} \upharpoonright
  \alpha$ does not depend on the choice of $y_s$, this implies $\pi_\infty \upharpoonright \alpha \in
  \HOD^{\dlmhat[H]}_{\dlm}$ and thus
  $\pi_\infty \upharpoonright \alpha \in \dlmhat$ by homogeneity of
  the forcing $\mathbb{P} = \Col(\omega, {<}\kappa_\infty)$.
\end{proof}
Now let $A \in V_{\deltadlm}^{\HODMn}$ be arbitrary. Let
$\alpha < \deltadlm$ be such that $A \subset \alpha$ is defined over
$M_n(x)[g]$ by a formula $\varphi$ with ordinal parameters from
$t \in [\Ord]^{<\omega}$ and let $\beta < \alpha$ be arbitrary. That
means $\beta \in A$ iff $M_n(x)[g] \vDash \varphi(\beta,t)$. Lemma \ref{lem:dmr} yields that this is
the case iff $\dlmhat[H] \vDash \varphi(\beta^*, t^*)$. Since
$\beta < \alpha < \deltadlm$, we have that
$\beta^* = \pi_\infty(\beta)$ by Lemma \ref{lem:piinftyid}. Moreover,
we have by Claim \ref{cl:*} that
$\pi_\infty \upharpoonright \alpha \in \dlmhat$. Therefore, it follows
by homogeneity of the forcing
$\mathbb{P} = \Col(\omega, {<}\kappa_\infty)$ that $A \in \dlmhat$
since $t^*$ is a fixed parameter in $\dlmhat$. Thus
$A \in V_{\deltadlm}^{\dlm}$, as desired.
\end{proof}

\section{The full $\HOD$ in $M_n(x)[g]$}\label{sec:fullHOD}

To compute the full model $\HODMn$, i.e., prove Theorem
\ref{thm:main}, we first show the following lemma. Recall that $\omega_2^{M_n(x)[g]} = \delta_\infty$.

\begin{lemma}\label{lem:HODeqMn(A)}
  $\HODMn = M_n(A)$ for some set $A \subseteq \omega_2^{M_n(x)[g]}$ with $A \in \HODMn$.
\end{lemma}

\begin{proof}
  Let $\bV$ denote the Vop\v{e}nka algebra in $M_n(x)[g]$ for making a
  real generic over $\HODMn$. By Vop\v{e}nka's theorem (see for example Theorem $15.46$ in \cite{Je03} or Theorem $9.0.1$ in
  \cite{La17}) there is a $\bV$-generic $G_x$ over $\HODMn$
  such that $x \in \HODMn[G_x]$ and in fact
  $\HODMn[G_x] = \HODMn_x$.
  
  \begin{claim}
    There is some $\tilde\bV \in \HODMn$ which is isomorphic to $\bV$ and a subset of $\omega_2^{M_n(x)[g]}$.\footnote{We would like to thank the anonymous referee for pointing out that this was overlooked in an earlier version of this article and for suggesting the argument we provide here.}
  \end{claim}
  
  \begin{proof}
    Work in $M_n(x)[g]$. Each real, i.e., element of $\Pot(\omega)$, can be coded by a countable ordinal and each set of reals can be coded by an ordinal $< \omega_2$. Forcing with the Vop\v{e}nka algebra $\bV$ is $\omega_2$-c.c. in $M_n(x)[g]$ as otherwise there would be an $\omega_2$ sequence of pairwise distinct non-empty sets of reals, contradicting $\CH$. The Vop\v{e}nka algebra is in $\HODMn$ and when considering $\HODMn[G_x]$ cardinals $\geq (\kappa^+)^{M_n(x)}$ are preserved. Since $(\kappa^{+\omega})^{M_n(x)}$ is below the least measurable cardinal of $M_n(x)$, $M_n(x)|(\kappa^{+\omega})^{M_n(x)}$ can be written as the $Lp^n$-stack of height $(\kappa^{+\omega})^{M_n(x)}$ above $x$ and is therefore by the argument in Lemma \ref{lem:closureSn} an element of $\HODMn_x = \HODMn[G_x]$. But the Vop\v{e}nka algebra is a subset of some ordinal $\alpha < (\kappa^{++})^{M_n(x)} = (\kappa^{++})^{\HODMn[G_x]} = (\kappa^{++})^{\HODMn}$, so there is some $\tilde\bV \in \HODMn$ which is isomorphic to $\bV$ and a subset of $(\kappa^+)^{M_n(x)}$.
  \end{proof}

  For the rest of this proof we write $\bV$ for the $\tilde\bV$ from the previous claim and show that $M_n(\bV) = \HODMn$.
 
  \begin{claim}
    $G_x$ is $\bV$-generic over $M_n(\bV)$.
  \end{claim}

  \begin{proof}
    The dense sets in question are elements of $\Pot(\bV)^{M_n(\bV)}$ and hence elements of $Lp^n(\bV) = \bigcup \{N \, | \, N \text{ is a countable } \bV\text{-premouse with } \rho_\omega(N) = \bV \text{ which is } n\text{-small, sound, and }(\omega,\omega_1,\omega_1)\text{-iterable}\}$. As $\bV \in \HODMn$, Lemma \ref{lem:closureSn} yields that $Lp^n(\bV) \in \HODMn$, which implies the claim.
  \end{proof}

  Let $\lambda$ be the least inaccessible of $M_n(x)$ above $\kappa$.

  \begin{claim}\label{cl:same2ndinacc}
    $V_{\lambda}^{M_n(x)} = V_{\lambda}^{M_n(\bV)[G_x]}$ and $\lambda$ is a cardinal in $M_n(\bV)[G_x]$.
  \end{claim}
  \begin{proof}
    Write $\kappa_0^{M_n(x)}$ and $\kappa_0^{M_n(\bV)}$ for the least measurable cardinal of $M_n(x)$ and $M_n(\bV)$ respectively. $M_n(\bV)$ and $M_n(\bV)[G_x]$ have the same least measurable cardinal. The proof of Lemma \ref{lem:closureSn} shows that $Lp^n(z) \in M_n(\bV)[G_x]$ for any $z \in V_{\delta_0}^{M_n(\bV)[G_x]}$, where $\delta_0$ denotes the least Woodin cardinal in $M_n(\bV)[G_x]$. As $M_n(x)|\kappa_0^{M_n(x)}$ is equal to the $Lp^n$-stack of height $\kappa_0^{M_n(x)}$ over $x$, it follows that \[ M_n(x)|\kappa_0^{M_n(x)} \subseteq M_n(\bV)[G_x]. \] Analogously, $M_n(\bV)|\kappa_0^{M_n(\bV)} \subseteq M_n(x)$ and in fact \[ M_n(\bV)[G_x]|\kappa_0^{M_n(\bV)} \subseteq M_n(x). \]

We are left with showing that $\lambda < \kappa_0^{M_n(\bV)}$ (as $\lambda < \kappa_0^{M_n(x)}$ is obvious). So suppose toward a contradiction that $\lambda \geq \kappa_0^{M_n(\bV)}$. Then \[ V_{\kappa_0^{M_n(\bV)}}^{M_n(\bV)[G_x]} = V_{\kappa_0^{M_n(\bV)}}^{M_n(x)}. \] Therefore, $\kappa_0^{M_n(\bV)}$ is not only a limit of inaccessible cardinals in $M_n(\bV)[G_x]$ but also in $M_n(x)$, contradicting our assumption that $\lambda \geq \kappa_0^{M_n(\bV)}$.
  \end{proof}


  \begin{claim}\label{cl:0}
    $V_\lambda^{M_n(\bV)} \subseteq \HODMn$.
  \end{claim}
  \begin{proof}
    This follows from the proof of Lemma \ref{lem:closureSn} as $M_n(\bV)|\lambda$ can be obtained as the $Lp^n$-stack of height $\lambda$ over $\bV$ and $\bV \in \HODMn$.
  \end{proof}
  
  Now we can show that the lemma holds below $\lambda$.

  \begin{claim}\label{cl:HODagreementuptolambda}
    $V_\lambda^{M_n(\bV)} = V_\lambda^{\HODMn}$.
  \end{claim}

  \begin{proof}
    We first show that $V_\lambda^{M_n(\bV)}[G_x] =
    V_\lambda^{\HODMn}[G_x]$. The inclusion $\subseteq$ follows from Claim
    \ref{cl:0}. For the other inclusion we have that
    \[\HOD^{M_n(x)[g]}[G_x] = \HOD_x^{M_n(x)[g]} \subseteq
    \HOD_x^{M_n(x)} \subseteq M_n(x), \] using
    the homogeneity and ordinal definability of the forcing
    $\Col(\omega, {<}\kappa)$. Therefore by Claim \ref{cl:same2ndinacc}
    \[V_\lambda^{\HODMn}[G_x] \subseteq V_\lambda^{M_n(x)} = V_\lambda^{M_n(\bV)}[G_x]. \] Finally, we argue that the equality
    $V_\lambda^{M_n(\bV)}[G_x] = V_\lambda^{\HODMn}[G_x]$ also holds true
    without adding the generic $G_x$. As by Claim \ref{cl:0} we have
    $V_\lambda^{M_n(\bV)} \subseteq V_\lambda^{\HODMn}$, we are again left
    with proving the other inclusion. Let $\bP = \bV \times \Col(\omega,{<}\kappa)$.
    Then $(G_x,g)$ is
    $\bP$-generic over both $V_\lambda^{M_n(\bV)}$ and $V_\lambda^{\HODMn}$,
    and $V_\lambda^{M_n(\bV)}[G_x,g] = V_\lambda^{\HODMn}[G_x,g]$. Let
    $a \in V_\lambda^{\HODMn}$ be a set of ordinals. Then there is a
    $\bP$-name $\sigma \in V_\lambda^{M_n(\bV)}$ such that
    $\sigma_{(G_x,g)} = a$. This is forced over $V_\lambda^{\HODMn}$,
    i.e., there is a $p \in \bP$ such that
    $V_\lambda^{\HODMn} \vDash \text{``}p \forces{}{} \sigma =
    \check{a}$''. Thus $V_\lambda^{M_n(\bV)}$ can compute the elements of
    $a$ using the forcing relation for $\bP$ below $p$. Hence
    $a \in V_\lambda^{M_n(\bV)}$, as desired.
  \end{proof}

  Now we are able to extend Claim \ref{cl:same2ndinacc} to the full models.

  \begin{claim}\label{cl:1}
    $M_n(\bV)[G_x] = M_n(x)$.
  \end{claim} 
  \begin{proof}
    Consider $M_n(x)[g]$ as a $V_\lambda^{M_n(x)[g]}$-premouse and note that it equals $M_n(V_\lambda^{M_n(x)[g]})$. We use $\cP^{M_n(x)[g]}(M_n(\bV)|\lambda)$ to denote the result of
    a $\cP$-construction in the sense of \cite{SchSt09} above
    $M_n(\bV)|\lambda$ inside the $V_\lambda^{M_n(x)[g]}$-premouse $M_n(x)[g]$. By Claim
    \ref{cl:same2ndinacc}, $V_\lambda^{M_n(\bV)}[G_x] = V_\lambda^{M_n(x)}$, so
    $V_\lambda^{M_n(\bV)}[G_x][g] = V_\lambda^{M_n(x)[g]}$ and this
    $\cP$-construction is well-defined. Moreover, the following
    argument shows that the construction never projects across
    $\lambda$.

    Assume toward a contradiction that there is a level $\cP$ of the
    $\cP$-construction above $M_n(\bV)|\lambda$ inside $M_n(x)[g]$ such
    that $\rho_\omega(\cP) = \rho < \lambda$. That means there is an
    $r\Sigma_{k+1}(\cP)$-definable set $a \subseteq \rho$ for some
    $k < \omega$ such that $a \notin \cP$. As by the proof of Claim
    \ref{cl:0}, $M_n(\bV)|\lambda \in \HODMn$ it follows by
    definability of the $\cP$-construction and of the extender
    sequence of $M_n(V_\lambda^{M_n(x)[g]})$ (see Lemma $1.1$ in \cite{Sch06} due to J. Steel) that
    $\cP \in \HODMn$. This means that in particular $a \in
    \HODMn$. But $a \subseteq \rho < \lambda$ and by Claim
    \ref{cl:HODagreementuptolambda},
    $V_\lambda^{\HODMn} = V_\lambda^{M_n(\bV)} = V_\lambda^\cP$, so
    $a \in \cP$. Contradiction.

    Now it follows by construction (see \cite{SchSt09}) that
    \[\cP^{M_n(x)[g]}(M_n(\bV)|\lambda)[G_x][g] = M_n(x)[g].\] But this
    yields that $\cP^{M_n(x)[g]}(M_n(\bV)|\lambda)[G_x] = M_n(x)$,
    without adding the generic $g$, by an argument as the one at the end of
    the previous claim. Moreover,
    $\cP^{M_n(x)[g]}(M_n(\bV)|\lambda) = M_n(\bV)$ and thus
    $M_n(\bV)[G_x] = M_n(x)$, as desired.
  \end{proof}

  This argument also shows the following claim.

  \begin{claim}\label{cl:0+}
     $M_n(\bV) \subseteq \HOD^{M_n(x)[g]}$.
   \end{claim}

  Now, the next claim follows from the first half of the proof of Claim \ref{cl:HODagreementuptolambda}.

  \begin{claim}\label{cl:2}
    $M_n(\bV)[G_x] = \HOD^{M_n(x)[g]}[G_x]$.
  \end{claim}

  Finally, the statement of Claim \ref{cl:2} also holds true without
  adding the generic $G_x$ by the argument at the end of the proof of
  Claim \ref{cl:HODagreementuptolambda}. Hence
  $M_n(\bV) = \HOD^{M_n(x)[g]}$, as desired.
\end{proof}

\begin{cor}\label{cor:HODF}
  Let $F(s) = s^*$ for $s \in [\Ord]^{<\omega}$. Then
  \[\HODMn = M_n(\dlm|\deltadlm, F \upharpoonright
    \deltadlm). \]
\end{cor}

\begin{proof}
  Note that $\dlm|\deltadlm$ and $F \upharpoonright \deltadlm$ are elements of $\HODMn$ by construction. Let $\eta = \sup F \pwimg \deltadlm$ and let $\gamma$ be the least inaccessible cardinal of $M_n(\dlm|\deltadlm, F \upharpoonright \deltadlm)$ above $\eta$. Let $A \subseteq \omega_2^{M_n(x)[g]}$ be as
  in the statement of Lemma \ref{lem:HODeqMn(A)},
  i.e., such that $\HODMn = M_n(A)$. Moreover, let
  $\varphi$ be a formula defining $A$, i.e., $\xi \in A$ iff $M_n(x)[g]
  \vDash \varphi(\xi)$. Then, as $F(\xi) = \pi_\infty(\xi)$ for $\xi < \delta_\infty$ by Lemma \ref{lem:piinftyid},
  \begin{align*}
    \xi \in A \text{ iff } M_n(\dlm | \deltadlm) \vDash \text{``}1&
    \forces{\bP}{} \varphi(\pi_\infty(\xi)), \text{where } \pi_\infty \text{ is the direct limit}\\ &\text{embedding from the systems on }\dlm\text{''} 
  \end{align*}
  for $\bP = \Col(\omega, {<}\kappa_\infty)$. Consider
  $L[E](\dlm | \deltadlm)^{M_n(\dlm|\deltadlm, F \upharpoonright
    \deltadlm)}$, the result of a fully backgrounded
  extender construction in the sense of \cite{MS94} inside
  $M_n(\dlm|\deltadlm, F \upharpoonright \deltadlm)$ above
  $\dlm | \deltadlm$. The premice $M_n(\dlm | \deltadlm)$ and
  $L[E](\dlm | \deltadlm)^{M_n(\dlm|\deltadlm, F \upharpoonright
    \deltadlm)}$ successfully
  compare to a common proper class premouse without drops on the main
  branches. Since the iterations take place above $\deltadlm$,
  $\xi < \deltadlm$ is not moved and we have by elementarity
  \begin{align*}
    \xi \in A \text{ iff } L[E]&(\dlm |
    \deltadlm)^{M_n(\dlm|\deltadlm, F \upharpoonright
      \deltadlm)} \vDash \text{``}1 \forces{\bP}{}
   \varphi(\pi_\infty(\xi)), \text{where } \pi_\infty\\ &\text{is the direct limit embedding from the systems on }\dlm\text{''}. 
  \end{align*}
 Therefore it follows that
 $A \in M_n(\dlm|\deltadlm, F \upharpoonright \deltadlm)$.

 By the same argument as in the proof of Claim \ref{cl:same2ndinacc} in the proof of Lemma \ref{lem:HODeqMn(A)} we now obtain that $\gamma$ is also the least inaccessible cardinal above $\eta$ of $M_n(A)$ and the universes of $M_n(\dlm|\deltadlm, F \upharpoonright \deltadlm)$ and $M_n(A)$ agree up to $\gamma$. In particular, we can rearrange
  $M_n(\dlm|\deltadlm, F \upharpoonright \deltadlm)$ and
  $M_n(A)$ as
  $V_\gamma^{M_n(\dlm|\deltadlm, F \upharpoonright
  \deltadlm)}$-premice. As such it follows that the
  following equalities for classes (not structures) hold:
  \begin{align*}
    M_n(\dlm|\deltadlm, F \upharpoonright \deltadlm) & =
                                                                 M_n(V_\gamma^{M_n(\dlm|\deltadlm, 
                                                                 F
                                                                 \upharpoonright
                                                                 \delta)})
    \\  & = M_n(A) = \HODMn. 
  \end{align*}
\end{proof}

The following corollary follows immediately from Lemma
\ref{lem:piinftyid} and Corollary \ref{cor:HODF}.

\begin{cor}\label{cor:HODpiinfty}
  $\HODMn = M_n(\dlm|\deltadlm, \pi_\infty \upharpoonright
  \deltadlm)$.
\end{cor}

We now consider the iteration strategy for $\dlm$. Let $\Lambda$ be
the restriction of $\Sigma_{M_{n+1}^{-}}$ to correctly guided finite
stacks $\vec{\cT}$ on $\dlm|\deltadlm$ such that
$\vec{\cT} \in \dlmhat|\kappa_\infty$, where $\kappa_\infty$ is the
least inaccessible cardinal in $\dlmhat$ above
$\deltadlm$.

\begin{lemma}\label{lem:LambdaHOD}
  $\Lambda \in \HODMn$. 
\end{lemma}

\begin{proof}
  Let $\cT$ be a maximal tree on $\dlm | \deltadlm$ with $\cT \in
  \dlmhat|\kappa_\infty$. Moreover, let $b = \Lambda(\cT)$. Let $\cR = \cM_b^\cT$ be the last model of $\cT^\smallfrown b$. Then
  $\cR \in \HODMn$. Moreover, let $\deltadlm^{\cF^*}$ be the least Woodin
  cardinal in $\dlm^{\cF^*}$, the direct limit of the system
  $\cF^{*,M(y)}$ for some/all dlm-suitable $M(y)$ in $\dlmhat[H]$. Then
  $\dlm^{\cF^*} | \deltadlm^{\cF^*}$ is an iterate of $\cR$. As
  $\pi_\infty \upharpoonright \deltadlm \in \HODMn$, we can identify
  $b$ inside $M_n(x)[g]$ as the unique branch through $\cT$ which is
  $(\pi_\infty \upharpoonright \deltadlm)$-realizable, i.e., such that
  there is an elementary embedding
  $\sigma: \cM_b^\cT \rightarrow \dlm^{\cF^*} | \deltadlm^{\cF^*}$
  with $\pi_\infty \upharpoonright \deltadlm = \sigma \circ i_b^\cT$.

  The same argument applies to pseudo-normal iterates $\cN$ of $\dlm$
  with $\cN|\delta^{\cN} \in \dlmhat|\kappa_\infty$ and maximal
  iteration trees $\cT$ on $\cN | \delta^{\cN}$ such that
  $\cT \in \dlmhat|\kappa_\infty$, hence $\Lambda \in \HODMn$.
\end{proof}

Similarly to Lemma $3.47$ in \cite{StW16} we finally need a method of
Boolean-valued comparison. As the proof is analogous we omit it.

\begin{lemma}\label{lem:BooleanComp}
  Let $H$ be $\Col(\omega, {<}\kappa_\infty)$-generic over $\dlmhat$,
  and let $\cQ$ be such that $\dlmhat[H] \vDash \text{``}\cQ \text{ is
    countable and }n\text{-suitable''}$. Then there is an $\cR$ such
  that
  \begin{enumerate}
  \item $\cR$ is a pseudo-normal iterate of $\cQ$,
  \item $\cR$ is a $\Sigma_{M_{n+1}^-}$-iterate of
    $\dlm$, and
  \item $\cR \in \dlmhat$. 
  \end{enumerate}
\end{lemma}

Finally, we can finish the proof of Theorem \ref{thm:main}.

\begin{thm}\label{thm:finalmain}
  \[\HODMn = M_n(\dlm|\deltadlm, \pi_\infty \upharpoonright \deltadlm)
    = M_n(\dlmhat|\kappadlm, \Lambda).\]
\end{thm}

\begin{proof}
  $\HODMn = M_n(\dlm|\deltadlm, \pi_\infty \upharpoonright \deltadlm)$
  is Corollary \ref{cor:HODpiinfty}. Moreover, 
  $M_n(\dlm|\deltadlm, \pi_\infty \upharpoonright \deltadlm) =
  M_n(\dlmhat|\kappadlm, \Lambda)$ follows from Lemma
  \ref{lem:BooleanComp} as follows. First, $\Lambda \in M_n(\dlm|\deltadlm, \pi_\infty \upharpoonright \deltadlm) = \HODMn$ and $\dlmhat|\kappa_\infty \in M_n(\dlm|\deltadlm, \pi_\infty \upharpoonright \deltadlm)$ by considering the $Lp^n$-stack on $\dlm|\deltadlm$. The direct limit of $\cF^{*,M(y)}$ for some $M(y)$ is the same as
  the direct limit of all $\Lambda$-iterates of $\dlm$ which are an
  element of $\dlmhat|\kappadlm$ via the comparison
  maps. Moreover, we have that $\pi_\infty$ is the canonical direct
  limit map of this system and therefore definable from
  $\dlmhat | \kappadlm$ and $\Lambda$. So
  $\pi_\infty \upharpoonright \deltadlm \in M_n(\dlmhat|\kappadlm,
  \Lambda)$. Now, $M_n(\dlm|\deltadlm, \pi_\infty \upharpoonright \deltadlm) = M_n(\dlmhat|\kappadlm, \Lambda)$ follows analogous to the proof of Corollary \ref{cor:HODF}.
\end{proof}

Note that Theorem \ref{thm:HODuptoDelta} and Lemma \ref{lem:HODeqMn(A)} together imply Corollary \ref{cor:main}, i.e., that the $\GCH$ holds in $\HODMn$. Finally, most of the arguments we gave in this and the previous sections generalize with only small changes to more arbitrary canonical self-iterable inner models, e.g. $M_\omega$, $M_{\omega +42}$. We leave the details to the reader.

\bibliographystyle{abstract}
\bibliography{References} 

\end{document}